\documentclass[11pt,letter]{amsart}
\usepackage[utf8]{inputenc}
\usepackage{amsmath, calligra, mathrsfs}
\usepackage{enumerate, color, hyperref}
\usepackage{bbm}

\DeclareMathOperator{\Hom}{\mathscr{H}\text{\kern -3pt {\calligra\large om}}}

\newtheorem{thm}{Theorem}[section]
\newtheorem{prop}[thm]{Proposition}

\newtheorem{lemma}[thm]{Lemma}

\newtheorem{hyp}[thm]{Hypothesis}

\theoremstyle{definition}
\newtheorem{rmk}[thm]{Remark}
\newtheorem{defi}[thm]{Definition}
\newtheorem{ex}[thm]{Exemple}

\usepackage{amsfonts}
\usepackage{amsmath}
\usepackage{amssymb}
\usepackage{graphics}
\usepackage{amscd}
\usepackage[all]{xy}
\usepackage{enumitem}
\usepackage{todonotes}
\usepackage{stackrel}

\newcommand{\beqn}{\begin{eqnarray*}}
\newcommand{\eeqn}{\end{eqnarray*}}
\newcommand{\beqa}{\begin{eqnarray}}
\newcommand{\eeqa}{\end{eqnarray}}

\DeclareMathOperator{\N}{\mathbb{N}}
\DeclareMathOperator{\Q}{\mathbb{Q}}

\DeclareMathOperator{\Z}{\mathbb{Z}}
\DeclareMathOperator{\V}{\mathbb{V}}
\DeclareMathOperator{\W}{\mathbb{W}}
\DeclareMathOperator{\G}{\mathbb{G}}
\DeclareMathOperator{\C}{\mathbb{C}}
\DeclareMathOperator{\A}{\mathbb{A}}

\DeclareMathOperator{\GL}{\mathrm{GL}}
\newcommand{\cO}{\mathcal{O}}

\newcommand{\dP}{\mathfrak{p}}
\newcommand{\cX}{\mathcal{X}}
\newcommand{\dX}{\mathfrak{X}}

\newcommand{\cN}{\mathfrak{n}}

\newcommand{\dH}{\mathfrak{H}}
\newcommand{\cG}{\mathcal{G}}
\newcommand{\cI}{\mathcal{I}}
\newcommand{\Spec}{\mathrm{Spec}}
\newcommand{\Spf}{\mathrm{Spf}}
\newcommand{\cU}{\mathcal{U}}

\newcommand{\cM}{\mathfrak{m}}
\newcommand{\cF}{\mathcal{F}}

\newcommand{\cR}{\mathcal{R}}
\newcommand{\cH}{\mathcal{H}}
\newcommand{\cW}{\mathcal{W}}
\newcommand{\dW}{\mathfrak{W}}
\newcommand{\Sym}{\mathrm{Sym}}

\newcommand{\M}{\mathrm{M}}

\makeatletter
\newcommand\figcaption{\def\@captype{figure}\caption}
\newcommand\tabcaption{\def\@captype{table}\caption}
\makeatother

\usepackage[left=1in, right=1in, top=1in, bottom=1in]{geometry}

\title[Finite slope Triple product $p$-adic $L$-functions]{Finite slope Triple product $p$-adic $L$-functions over totally real number fields}
	
\author{Santiago Molina}

\newcommand{\Addresses}{{
		\bigskip
		\footnotesize
		
		Santiago Molina; Universitat Polit\`{e}cnica de Catalunya\\Campus Nord, Calle Jordi Girona, 1-3, 08034 Barcelona, Spain\par\nopagebreak
		\texttt{santiago-molina@upc.edu}
		
}}

\begin{document}
	
	\begin{abstract} 
	We construct $p$-adic $L$-functions associated with triples of finite slope $p$-adic families of quaternionic automorphic eigenforms over totally real fieldson Shimura curves. These results generalize a previous construction, joint work with D.Barrera, performed in the ordinary setting. 
	\end{abstract}
	
  \maketitle
  	

\section{Introduction}

The theory of $p$-adic L-functions has grown tremendously during the last decades due to its important arithmetic applications. In particular, it has been essential for recent developments towards the Birch and Swinnerton-Dyer conjecture and its generalizations.
The seminal work of Kato provided deep results on the Birch and Swinnerton-Dyer conjecture in rank $0$ for twists of elliptic curves over $\Q$ by Dirichlet characters. More recently, similar results with twists by certain Artin representations of dimension 2 and 4 have been obtained by Bertolini-Darmon-Rotger in \cite{BDR15} and Darmon-Rotger in \cite{DR17} using analogous methods. In \cite{LLZ16}, \cite{KLZ17}, \cite{LZ}, \cite{AI17} and \cite{BSV2}  the development of such methods has been extended in different directions, as for example bounding certain Selmer groups and treating finite slope settings. In all these situations the 
$p$-adic $L$-function attached to a triple of $p$-adic families of modular forms has played a very important role. Such triple product $p$-adic $L$-functions were constructed in \cite{harris-tilouine01}, \cite{DR14}, \cite{hsieh18} in the ordinary case and in \cite{AI17} for Coleman Families, namely, in the finite slope situation. 

This paper deals with the construction of triple product $p$-adic L-functions associated with families of quaternionic automorphic forms in Shimura curves over totally real fields in the finite slope situation. In a previous paper with D. Barrera \cite{BM}, we constructed such $p$-adic L-functions in the ordinary setting.
In fact, we were able to construct $p$-adic families of finite slope quaternionic automorphic forms over totally real fields, but we had to restrict ourselves to the ordinary setting for the construction of the corresponding triple product $p$-adic L-functions. The reason for such a restriction was the inability to iterate the Gauss-Manin connection in the finite slope situation. In this article we manage to overcome these obstacles by following the strategy of Andreatta and Iovita in \cite{AI17}. 

Our setting is in many aspects analogous to the setting of elliptic modular forms treated in \cite{AI17}, hence many of the techniques will be similar. 
The main novelty of this paper with respect to \cite{AI17} is the extension of the $\chi$-iteration of the Gauss-Manin connection to the universal character. This allows us to extend the definition of triple product $p$-adic L-function to the whole weight space, removing the crucial \cite[Assumption 4.7]{AI17}. This also improves and complete the results in \cite{AI17} when $F = \Q$.


Our strategy relies on defining the space of nearly overconvergent families using the theory of formal vector bundles also introduced in \cite{AI17}. These formal vector bundles have been also used in \cite{Gra} to perform similar constructions in case of Hilbert modular schemes. However, we define a new type of formal vector bundle, already used in \cite{BM} to define $p$-adic families of finite slope quaternionic automorphic forms (\S \ref{ss:formal vector bundles}). During the elaboration of this paper, Andreatta and Iovita reported to the author that in parallel they had reached the same new definition of formal vector bundle.
This space of nearly overconvergent families admits a good definition of a Gauss-Manin connection (Theorem \ref{GMthm}).
We use the explicit description of the universal character given in \S \ref{expdescunivchar} to define a formal power series of endomorphisms liable to be the universal iteration of the Gauss-Manin connection. The choice of this new type of formal vector bundle ensures the convergence of such power series, as verified using $q$-expansions on Serre-Tate coordinates.
As mentioned before, these computations also apply in the setting over $\Q$, hence as a consequence one can define triple product $p$-adic L-function in the finite slope situation over $\Q$ without any restriction on the input weights (\cite[Assumption 5.4]{AI17}), improving the results known so far.

\subsection{Main result} Let $F$ be a totally real number field with integer ring $\cO_F$ and fix a real embedding $\tau_0$.  Let $p> 2$ be a prime number and let $\mathfrak{p}_0$ the prime over $p$ associated to $\tau_0$ under a fixed embedding $\bar\Q\hookrightarrow\C_p$. During all this paper we will suppose that $F_{\mathfrak{p}_0}=\Q_p$.

Let $B$ be a quaternion algebra over $F$ that splits at any prime over $p$, and it is ramified at any real place but $\tau_0$. In \cite{BM}, we construct $p$-adic families of automorphic forms on $(B\otimes\A)^\times$ as global sections of certain $p$-adic sheaves on certain unitary Shimura curves, extending the previous work of \cite{brasca13}. 
Such families are parametriced by the $(d+1)$-dimensional Iwasawa algebra $\Lambda=\Z_p[[(\cO_{F}\otimes\Z_p)^\times\times\Z_p^{\times}]]$.

Let $\mu_1, \mu_2, \mu_3$ be three finite slope $p$-adic families of automorphic eigenforms for $B$ and we denote by $\Lambda_1$, $\Lambda_2$ and $\Lambda_3$ the rings over which they are defined, respectively. Let  $x, y, z$ be a triple of classical points corresponding to $((\underline{k}_1, \nu_1), (\underline{k}_2, \nu_2), (\underline{k}_3, \nu_3))$, where $\mu_i\in\N$ and $\underline{k}_i\in\N^{[F:\Q]}$. We denote by $\pi_x$ the automorphic representation of $(B\otimes\A_F)^\times$ generated by the automorphic form obtained from the specialization of $\mu_1$ at $x$, and $\Pi_x$ the corresponding cuspidal automorphic representation of $\mathrm{GL}_2(\A_F)$. We also denote by $\alpha_{x}^{\dP}$ and $\beta_{x}^{\dP}$ the roots of the Hecke polynomial at $\mathfrak{p}$. In the same way we obtain $\pi_y$, $\Pi_y$, $\alpha_{y}^{\dP}$, $\beta_{y}^{\dP}$, $\pi_z$, $\Pi_z$, $\alpha_{z}^{\dP}$, $\beta_{z}^{\dP}$. 

We have (see Lemma \ref{l:construction} and Theorem \ref{t:interpolation} for more details):
\begin{thm}\label{t:main theorem 2} There exists $\mathcal{L}_p(\mu_1,\mu_2,\mu_3)\in \Lambda_1\hat{\otimes}\Lambda_2\hat{\otimes}\mathrm{Frac}(\Lambda_3)$ such that for each classical point $(x, y, z)$ corresponding to a triple $((\underline{k}_1, \nu_1), (\underline{k}_2, \nu_2), (\underline{k}_3, \nu_3))$ we have:
$$\mathcal{L}_p(\mu_1,\mu_2,\mu_3)(x, y, z)= K\cdot\left(\prod_{\dP\mid p}\frac{\mathcal{E}_\dP(x,y,z)}{\mathcal{E}_{\dP,1}(z)}\right)\cdot L\left(\frac{1-\nu_1-\nu_2-\nu_3}{2},\Pi_{x}\otimes\Pi_{y}\otimes\Pi_{z}\right)^{\frac{1}{2}},$$
where  $K$ is a non-zero constant depending of $(x, y, z)$, $\mathcal{E}_\dP(x,y,z)=$
\[
\left\{\begin{array}{lc}
\mbox{\small$(1-\beta_{x}^{\dP}\beta_{y}^{\dP}\alpha_{z}^{\dP}\varpi_{\dP}^{-\underline{m}_{\dP}-\underline{2}})(1-\alpha_{x}^{\dP}\beta_{y}^{\dP}\beta_{z}^{\dP}\varpi_{\dP}^{-\underline{m}_{\dP}-\underline{2}})(1-\beta_{x}^{\dP}\alpha_{y}^{\dP}\beta_{z}^{\dP}\varpi_{\dP}^{-\underline{m}_{\dP}-\underline{2}})(1-\beta_{x}^{\dP}\beta_{y}^{\dP}\beta_{z}^{\dP}\varpi_{\dP}^{-\underline{m}_{\dP}-\underline{2}})$},&\dP\neq\dP_0\\
\mbox{\small$(1-\alpha_{x}^{\dP_0}\alpha_{y}^{\dP_0}\beta_{z}^{\dP_0}p^{1-m_{0}})(1-\alpha_{x}^{\dP_0}\beta_{y}^{\dP_0}\beta_{z}^{\dP_0}p^{1-m_{0}})(1-\beta_{x}^{\dP_0}\alpha_{y}^{\dP_0}\beta_{z}^{\dP_0}p^{1-m_{0}})(1-\beta_{x}^{\dP_0}\beta_{y}^{\dP_0}\beta_{z}^{\dP_0}p^{1-m_{0}})$},&\dP=\dP_0
\end{array}\right., 
\]
\[
\mathcal{E}_{\dP,1}(z):=\left\{\begin{array}{lc} 
(1- (\beta_z^{\dP})^2\varpi_{\dP}^{-\underline{k}_{3,\dP}-\underline{2}})\cdot (1- (\beta_z^{\dP})^2\varpi_{\dP}^{-\underline{k}_{3,\dP}-\underline{1}}),&\dP\neq \dP_0,\\
(1- (\beta_z^{\dP_0})^{2}p^{-k_{3,\tau_0}})\cdot (1- (\beta_z^{\dP_0})^{2}p^{1-k_{3,\tau_0}}),&\dP= \dP_0,
\end{array}\right. 
\]
$m_{0}=\frac{k_{1,\tau_0}+k_{2,\tau_0}+k_{3,\tau_0}}{2}\geq 0$, $\underline{m}_{\dP}=\frac{\underline{k}_{1,\dP}+\underline{k}_{2,\dP}+\underline{k}_{3,\dP}}{2}= \left(\frac{k_{1,\tau}+k_{2,\tau}+k_{3,\tau}}{2}\right)_{\tau \sim\dP}$ and $\tau \sim\dP$ means real embeddings $\tau$ corresponding to embeddings $F_\dP\hookrightarrow\C_p$ through $\iota_p$.
\end{thm}

As mentioned before this result is analogous to \cite[Theorem 1.2]{BM} where the triple product $p$-adic L-function $\mathcal{L}_p(\mu_1,\mu_2,\mu_3)$ is constructed in the ordinary setting. We have been able to extend this result to the finite slope setting thanks to the strategy to $p$-adically interpolate of the integral powers of the Gauss-Manin connexion. 
We perform such $p$-adic interpolations as power series of endomorphisms in a well chosen space of nearly overconvergent modular forms. 

\subsubsection*{Acknowledgements.}
The author is supported in part by DGICYT Grant MTM2015-63829-P.
This project has received funding from the European Research Council
(ERC) under the European Union's Horizon 2020 research and innovation
programme (grant agreement No. 682152).

\section{Notations}
Let $\mathbb{A}$ be the adeles of $\Q$ and $\A_f$ the finite adeles. Let $F$ be a totally real field of degree $d= [F:\Q]$, $\mathcal{O}_{F}$ its ring of integers and $\Sigma_F$ the set of real embeddings of $F$. In all this paper we fix an embedding  $\tau_0 \in \Sigma_F$. We denote by $\underline{1}\in \Z[\Sigma_F]$ the element with each coordinate equals to $1$. For $x\in F^\times$ and $\underline{k}\in \Z[\Sigma_F]$ we put $x^{\underline{k}}= \prod_{\tau\in\Sigma_F}\tau(x)^{k_{\tau}}$.

We fix a prime number $p> 2$, and we identinfy $\Sigma_F$ with the set of the embeddings of $F$ in $\overline{\Q}_p$ once we fix an embedding $\iota_p:\bar\Q\hookrightarrow\C_p$. For each prime of $\mathfrak{p} \mid p$ let $F_{\mathfrak{p}}$ be the completion of $F$ at $\mathfrak{p}$, $\Sigma_{\mathfrak{p}}$ the set of its embeddings  in $\C_p$, $\mathcal{O}_{\mathfrak{p}}$ its ring of integers, $\kappa_{\mathfrak{p}}$ its residue field, $q_{\mathfrak{p}}= \sharp\kappa_{\mathfrak{p}}$ and $e_{\mathfrak{p}}$ the ramification index. We also fix uniformizers $\varpi_{\mathfrak{p}} \in \mathcal{O}_{\mathfrak{p}}$. Moreover we use the notation $\cO:= \cO_{F}\otimes \Z_p$ which naturally decompose as $\cO= \prod_{\dP}\cO_\dP$. 
We will denote by $\mathfrak{p}_0$ the prime corresponding to $\iota_p$ and $\tau_0$. 
We suppose the following hypothesis:
\begin{hyp} \label{hypothesis 1} $[F_{\dP_0}: \Q_p]= 1$ and $F$ unramified at $p$. 
\end{hyp}
We denote $\cO^{\tau_0}:=\prod_{\mathfrak{p}\neq \mathfrak{p}_0}\cO_{\mathfrak{p}}$. Thus we have the following decomposition $\cO= \cO_{\dP_0}\times \cO^{\tau_0}= \Z_p\times \cO^{\tau_0}$.

 We also fix a quaternion algebra over $F$ denoted by $B$ such that:
\begin{itemize}
	\item[(i)] split at $\tau_0$ and at each $\mathfrak{p}\mid p$,
	\item[(ii)] is ramified at each $\tau \in \Sigma_F\setminus \{\tau_0\}$.
\end{itemize}

We choose from now on $\lambda \in \Q$ such that $\lambda < 0$ and $p$ split in $\Q(\sqrt{\lambda})$. Let $E:= F(\sqrt{\lambda})$ and denote $z \mapsto \overline{z}$ the not-trivial automorphism of $E/F$.

We denote $D:= B\otimes_{F}E$ which is a quaternion algebra over $E$ and we consider the involution on $D$ defined by $l= b\otimes z \mapsto \overline{l}:= \bar b\otimes \overline{z}$ where $\bar b$ is the canonical involution of $B$. We fix $\delta \in D^\times$ such that $\overline{\delta}= -\delta$ and define a new involution on $D$ by $l\mapsto l^{\ast}:= \delta^{-1}\overline{l}\delta$.  We denote by $V$ to the underlying $\Q$-vector space of $D$ endowed with the natural left action of $D$.  
We have a symplectic bilinear form on $V$: 
$$\Theta: V \times V \rightarrow \Q, \ \ \  (v,w) \mapsto \mathrm{Tr}_{E/\Q}(\mathrm{Tr}_{D/E}(v\delta w^{\ast})).$$

Let $G'$  be the reductive group over $\Q$ such that for each $\Q$-algebra $R$ we have:
$$G'(R)= \left\lbrace  \mathrm{D-linear \ symplectic \ similitudes \ of } \ (V\otimes_{\Q}R, \Theta\otimes_{\Q}R) \right\rbrace. $$

Let $A\subset B$ be a $\cO_F$-order, and let $A_D:=A\otimes_{\cO_F}\cO_E$.
We introduce a way to cut certain modules endowed with an action of $A_D$. We denote by $\psi:\Q(\sqrt{\lambda})\longrightarrow D$ the natural embedding given by $z\longmapsto1\otimes z$ and fix an extension $R/\cO_E$  such that $A\otimes_{\cO_F}R=\M_2(R)$. For any $R$-module $M$ endowed with a linear action of $A_D$, we define 
	$$M^{+}:=\left\{v\in M:\;\psi(e)\ast v=ev,\mbox{ for all }e\in \Z(\sqrt{\lambda})\right\},$$
	$$M^{-}:=\left\{v\in M:\;\psi(e)\ast v=\bar ev,\mbox{ for all } e\in \Z(\sqrt{\lambda})\right\}.$$
	 Each $M^{\pm}$ is equipped with an action of $A\otimes_{\Z}R=\M_2(R)\otimes_{\Z}\cO_F$ and we put $M^{\pm,1}:=(\begin{smallmatrix}1&0\\0&0\end{smallmatrix})M^{\pm}$ and $M^{\pm,2}:=(\begin{smallmatrix}0&0\\0&1\end{smallmatrix})M^{\pm}.$  
	Note that both are isomorphic $R\otimes_{\Z}\cO_F$-modules through the matrix $(\begin{smallmatrix}0&1\\1&0\end{smallmatrix})$. Moreover, by construction we have:
	 $$M\supseteq M^+\oplus M^-= M^{+,1}\oplus M^{-,1}\oplus M^{+,2}\oplus M^{-,2},$$
	 and the inclusion is an equality if ${\rm disc}(\Q(\sqrt{\lambda}))\in R^\times$.

\section{Unitary modular forms} We introduce unitary Shimura curves associated with $G'$ and their moduli interpretation. We define the sheaves which give rise to modular forms. See \cite[\S 4]{BM} for further details.

\subsection{Unitary Shimura curves}\label{ss:unitary Shimura curves} 
Let $\cO_B\subset B$ be a maximal order and let $\cO_D:=\cO_B\otimes_{\cO_F}\cO_E$. The order $\cO_D$ is maximal (locally) at every prime not dividing ${\rm disc}(B){\rm disc}(D)^{-1}$. We denote by $G_D$ the algebraic group attached to $\cO_D$, namely, $G_D(R):=(\cO_D\otimes_{\Z}R)^\times$ for any $\Z$-algebra $R$.

Let $\cN$ be an integral ideal of $F$ prime to $\mathrm{disc}(B)$ and consider the open compact subgroups of $G_D(\hat\Z)$ and $G'(\A_f)$: 
$$K_1^D(\cN ):=\left\{\left(\begin{smallmatrix}a&b\\c&d\end{smallmatrix}\right)\in G_D(\hat\Z):\; c\equiv d-1\equiv 0\;{\rm mod}\;\cN\cO_E\right\},\qquad K_{1}^{'}(\cN ):=K_1^D(\cN )\cap G'(\A_f).$$

We define the \emph{unitary Shimura curve} over $\C$ of level $K_1'(\cN )$ as:
\begin{equation}\label{ShimCurv}
X(\C):=G'(\Q)_+\backslash (\dH\times G'(\A_f)/K_{1}^{'}(\cN )).
\end{equation}

Let $\cO_{B,\cM}\subseteq \cO_B$ be an Eichler order with a well chosen level $\cM\mid\cN$ so that the conditions of \cite[Lemma 4.2]{BM} are satisfied. 
Let $L/E$ be a finite extension such that $B\otimes_F L=\M_2(L)$. By \cite[\S 2.3]{carayol86} the Riemann surface $X(\C)$ has a model (also denoted $X$) defined over $L$. This curve solves the following moduli problem:  if $R$ is a $L$-algebra then $X(R)$ corresponds to the set of the isomorphism classes of tuples $(A,\iota, \theta, \alpha)$ where: 
	\begin{itemize}
		\item[$(i)$] $A$ is an abelian scheme over $R$  of relative dimension $4d$. 
		\item[$(ii)$] $\iota: \cO_\cM \rightarrow \mathrm{End}_R(A)$ gives an action of the ring $\cO_\cM$ on $A$  such that ${\rm Lie}(A)^{-,1}$ is of rank 1 and the action of $\cO_F$ factors through $\cO_F\subset E\subseteq L$.
		\item[$(iii)$] $\theta$ is a $\cO_\cM$-invariant homogeneous polarization of $A$ such that the Rosati involution sends $\iota(d)$ to $\iota(d^\ast)$.
		\item[$(iv)$]  $\alpha$ is a class modulo $K_{1}^{'}(\cN)$ of $\cO_\cM$-linear symplectic similitudes $\alpha:\hat T(A)\stackrel{\simeq}{\rightarrow}\hat\cO_\cM$.
	\end{itemize}	


\begin{rmk}\label{rmkonpts} Since $p \nmid{\rm disc}(B)$, a class $\alpha$ of $\cO_\cM$-linear symplectic similitudes modulo $K_{1}^{'}(\cN)$ is decomposed as $\alpha=\alpha_p\times\alpha^p$. By \cite[Remark 4.5]{BM} to provide a $\alpha_p$ modulo $K_{1}^{'}(\cN)_p$ amounts to giving a point  $P\in A[\cN\cO_{F}\otimes\Z_p]^{-,1}$ that generates a subgroup isomorphic to $(\cO_{F}\otimes\Z_p)/(\cN\cO_{F}\otimes\Z_p)$. We have an analogous description in case of $\Gamma_0(\cN)$-structures.
	
\end{rmk}

\subsection{Modular sheaves}\label{ss:modular sheaves} We introduce the sheaves which give rise to the modular forms for $G'$. Let $L_0/F$ be an extension such that $B\otimes_{F}L_0= \mathrm{M}_2(L_0)$, write $L=L_0(\sqrt{\lambda})\supset E$, and denote by  $X_L$ the  base change to $L$ of the unitary Shimura curve $X$. Using the universal abelian variety $\pi: A\rightarrow  X_L$ we define the following coherent sheaves on $X_L$:
\[
\omega:=\left(\pi_\ast\Omega^1_{A/X_L}\right)^{+,2} \qquad \omega_-:=\left(\left(R^1\pi_\ast\cO_A\right)^{+,2}\right)^\vee \qquad \cH:=\left(\cR^1\pi_\ast\Omega^\bullet_{A/X_L}\right)^{+,2}.
\]
The sheaf  $\cH$ is endowed with a Gauss-Manin connection $\bigtriangledown:\cH\rightarrow\cH\otimes\Omega^1_{X_L}$. The natural $\cO_D$-equivariant exact sequence:
	\begin{equation}\label{exseqdR}
	0\longrightarrow \pi_\ast\Omega^1_{A/X_L} \longrightarrow \cR^1\pi_\ast\Omega^\bullet_{A/X_L} \longrightarrow R^1\pi_\ast\cO_A \longrightarrow 0,
	\end{equation}
induces the Hodge exact sequence (see \cite[\S 2.3.1]{ding17})
$0\rightarrow \omega \rightarrow\cH \rightarrow\omega_-^\vee \rightarrow 0$.

If $L$ contains the Galois closure of $F$ then the natural decomposition $F\otimes_{\Q}L\cong L^{\Sigma_F}$ induces:
$\omega= \bigoplus_{\tau \in \Sigma_F}\omega_{\tau}$  and $\cH= \bigoplus_{\tau \in \Sigma_F}\cH_{\tau}$.
As the sheaves $(\Omega^1_{A/X_L})^{+,1}$ and $(\Omega^1_{A/X_L})^{+,2}$ are isomorphic then condition $(ii)(2)$ of the moduli problem of $X$ imply that $\omega_{\tau_0}$ is locally free of rank 1, while $\omega_\tau$ is of rank $2$ for $\tau \neq \tau_0$.  Moreover, $\omega_-$ is locally free of rank 1. Thus, $\omega_\tau=\cH_\tau$ for each $\tau\neq \tau_0$, and we have the exact sequence
\begin{equation}\label{Hodge1}
0\longrightarrow \omega_{\tau_0} \longrightarrow\cH_{\tau_0} \stackrel{\epsilon}{\longrightarrow}\omega_-^{\vee} \longrightarrow 0.
\end{equation}

Let $\underline{k}=(k_\tau)\in \N[\Sigma_F]$, we introduce the \emph{modular sheaves} over $X_L$ considered in this paper:
$$\omega^{\underline{k}}:= \omega_{\tau_0}^{\otimes k_{\tau_0}}\otimes \bigotimes_{\tau \neq \tau_0}\Sym^{k_\tau}\omega_\tau.$$

\begin{defi} A \emph{modular form} of weight $\underline{k}$ and coefficients in $L$ for $G'$ is a global section of $\omega^{\underline{k}}$, i.e. an element of $\mathrm{H}^{0}(X_L, \omega^{\underline{k}})$. 
\end{defi}

In \cite[\S 4.3]{BM} an isomorphism
$\varphi_{\tau_0}:\omega_{\tau_0}\stackrel{\simeq}{\longrightarrow}\omega_-$
is provided by the polarization. Thus, the Kodaira-Spencer isomorphism (see \cite[Lemme 2.3.4]{ding17}) induces the isomorphism:
\[
KS:\Omega^1_{X_L}\stackrel{\simeq}{\longrightarrow} \omega_{\tau_0}\otimes\omega_-\stackrel{\varphi_{\tau_0}^{-1}}{\longrightarrow} \omega_{\tau_0}^{\otimes 2}.
\] 

If $\tau\neq\tau_0$, also the polarization provides an isomorphism:
\begin{equation}\label{isowedge}
\varphi_\tau:\bigwedge^2\left(\cR^1\pi_\ast\Omega^\bullet_{A/X_L}\right)_\tau^{+,2} =\bigwedge^2 \omega_\tau\stackrel{\simeq}{\longrightarrow}\cO_{X_L}
\end{equation}

\subsection{Katz Modular forms}\label{KatzModForm}
Let $R_0$ be a $L$-algebra and let $f\in H^0(X/R_0,\omega^{\underline{k}})$. If $R$ is a $R_0$-algebra, $(A,\iota,\theta,\alpha)$ is a tuple corresponding to a point of $X(\Spec(R))$ and $w=(f_{0},(f_\tau,e_\tau)_{\tau\neq \tau_0})$ is a $R$-basis of $\omega_A=\left(\Omega^1_{A/R}\right)^{+,2}$, then there exists $f(A,\iota,\theta,\alpha,w)\in\bigotimes_{\tau\neq \tau_0}\Sym^{k_\tau}\left(R^2\right)^\vee$ such that
\[
f(A,\iota,\theta,\alpha)=f(A,\iota,\theta,\alpha,w)(P)f_{0}^{\otimes k_{\tau_0}}\mbox{ with } P((x_\tau, y_\tau)_{\tau\neq\tau_0})=\bigotimes_{\tau\neq {\tau_0}}|\begin{smallmatrix}f_\tau&e_\tau\\x_\tau&y_\tau\end{smallmatrix}|^{k_\tau}\varphi_{\tau}(f_\tau\wedge e_\tau)^{-k_\tau}
\]

Thus a section $f\in H^0(X/R_0,\omega^{\underline{k}})$ is characterized as a rule that assigns to any $R_0$-algebra $R$ and $(A,\iota,\theta,\alpha, w)$ over $R$ a linear form 
\[
f(A,\iota,\theta,\alpha,w)\in\bigotimes_{\tau\neq {\tau_0}}\Sym^{k_\tau}\left(R^2\right)^\vee
\]
satisfying:
\begin{itemize}
\item[(A1)] The element $f(A,\iota,\theta,\alpha,w)$ depends only on the $R$-isomorphism class of $(A,\iota,\theta,\alpha)$.

\item[(A2)] Formation of $f(A,\iota,\theta,\alpha,w)$ commutes with extensions $R\rightarrow R'$ of $R_0$-algebras.

\item[(A3)] If $(t, \underline{g})\in R^\times\times\GL_2(R)^{\Sigma_F\setminus\{\tau_0\}}$:
\[
f(A,\iota,\theta,\alpha,w(t,\underline{g}))=t^{-k_{\tau_0}}\cdot\left(\underline{g}^{-1} f(A,\iota,\theta,\alpha,w)\right).
\]
where $(t,\underline{g})$ acts on $w$ in a natural way (see \cite[\S 4.4]{BM} for further details).
\end{itemize}

\section{Overconvergent families of modular forms}\label{s:integral models and canonical groups} In this section we construct families of modular sheaves that $p$-adically interpolate the modular sheaves introduced in \ref{ss:modular sheaves}. Again, for further details see \cite[\S 5]{BM}.

\subsection{Integral models, Hasse invariants and canonical subgroups}\label{ss:integral models and divisible groups} Let $\cN$ be an integral ideal of $F$ prime to $p$ and $\mathrm{disc}(B)$. Write $K_1'\left(\mathfrak{n},\prod_{\dP\neq\dP_0}\dP\right):=K_1'\left(\mathfrak{n}\right)\cap K_0'\left(\prod_{\dP\neq\dP_0}\dP\right)$, and denote by $X$ the unitary Shimura curve of level $K_1'\left(\mathfrak{n},\prod_{\dP\neq\dP_0}\dP\right)$ similarly as in \S\ref{ss:unitary Shimura curves}.  Since $p$ is coprime to the discriminant of $B$, by \cite[\S 5.3]{carayol86} $X$ admits a canonical model $X_{\mathrm{int}}$ over $\cO_{\dP_0}$, representing the analogue moduli problem described in \S\ref{ss:unitary Shimura curves} but exchanging an $E$-algebra by an $\cO_{\dP_0}$-algebra $R$. Namely, it classifies quadruples $(A,\iota,\theta,\alpha^{\dP_0})$ over $R$, where $\alpha^{\dP_0}$ is a class of $\cO_D$-linear symplectic similitudes outside $\dP_0$. By \cite[\S 5.4]{carayol86}, $X_{\mathrm{int}}$ has good reduction. Let $\pi: {\bf A} \rightarrow X_{\mathrm{int}}$ be the universal abelian variety, whose existence is guaranteed by the moduli interpretation of $X_{\mathrm{int}}$. Notice that the universal abelian variety $A\rightarrow X_L$ introduced in \S \ref{ss:modular sheaves} is the generic fibre of ${\bf A} \rightarrow X_{\mathrm{int}}$. Since we have added $\Gamma_0(\dP)$-structure for all $\dP\neq\dP_0$, ${\bf A}$ is endowed with a subgroup $C_\dP\subset {\bf A}[\dP]^{-,1}$ isomorphic to $\cO_\dP/\dP$ by Remark \ref{rmkonpts}. 

  Let $\dX$ be denote the formal scheme over $\Spf(\cO_{\dP_0})$ obtained as the completion of $X_{\mathrm{int}}$ along its special fiber which is denoted by  $\bar{X}_{\mathrm{int}}$ .

The $p$-divisible group ${\bf A}[p^\infty]$ over $X_{\mathrm{int}}$ is decomposed as:
\[
{\bf A}[p^\infty]={\bf A}[\mathfrak{p}_0^\infty]^+\oplus \left[ \bigoplus_{\mathfrak{p} \neq \mathfrak{p}_0}{\bf A}[\mathfrak{p}^\infty]^+\right] \oplus {\bf A}[\mathfrak{p}_0^\infty]^-\oplus \left[\bigoplus_{\mathfrak{p} \neq \mathfrak{p}_0}{\bf A}[\mathfrak{p}^\infty]^-\right] ,
\] 
We are interested in the $p$-divisible groups $\mathcal{G}_0:= {\bf A}[\mathfrak{p}_0^\infty]^{-, 1}$ and $\mathcal{G}_{\mathfrak{p}}:= {\bf A}[\mathfrak{p}^\infty]^{-, 1}$ if $\mathfrak{p} \neq \mathfrak{p}_0$, which are defined over $X_{\mathrm{int}}$ and endowed with actions of $\cO_{\dP_0}$ and $\cO_{\mathfrak{p}}$ respectively. The sheaves of invariant differentials of the corresponding Cartier dual $p$-divisible groups are denoted by $\omega_{0}:= \omega_{\cG_{0}^D}$ and  $\omega_{\mathfrak{p}}:= \omega_{\cG_{\mathfrak{p}}^D} $ if $\mathfrak{p} \neq \mathfrak{p}_0$. This fits with the definition of $\omega$ given in \S \ref{ss:modular sheaves} since, as explained in \cite[\S 5.1]{BM}, $\left(\pi_\ast\Omega^1_{{\bf A}/X_{\rm int}}\right)^{+,2}\simeq\left(\pi_\ast\Omega^1_{{\bf A}^\vee/X_{\rm int}}\right)^{-,1}$ by means of the polarization.

As explained in \cite[\S 5.1]{BM}, the universal polarization provides an isomorphism of sheaves of invariant differentials $\omega_0\simeq\omega_{\cG_0}$ compatible with $\varphi_{\tau_0}$. Moreover, for any ring $\hat\cO$ containing all the $p$-adic embeddings of $\cO_F\hookrightarrow\cO_{\C_p}$, if we extend our base ring $\cO_{\dP_0}$ to $\hat\cO$ then we have a decomposition $\omega_{\mathfrak{p}}= \oplus_{\tau \in \Sigma_\mathfrak{p}}\omega_{\mathfrak{p}, \tau}$, where each $\omega_{\mathfrak{p}, \tau}$ has rank $2$ and we have an isomorphism $\varphi_\tau:\bigwedge^2\omega_{\dP,\tau}\stackrel{\simeq}{\rightarrow}\cO_{X_{\rm int}}$. The sheaves $\omega_{\dP,\tau}$ and $\omega_0$ are the integral models of $\omega_\tau$ and $\omega_{\tau_0}$, respectively.

There is a dichotomy in $X_{\mathrm{int}}$ which says that any point in the generic fiber $\bar{X}_{\mathrm{int}}$ is ordinary or  supersingular (with respect to $\mathcal{G}_0$), and there are finitely many supersingular points in $\bar{X}_{\mathrm{int}}$. From \cite[Proposition 6.1]{kassaei04} there exists $\mathrm{Ha}\in H^0\left(\bar{X}_{\mathrm{int}},\omega_{0}^{p-1}\right)$ that vanishes exactly at supersingular geometric points and these zeroes are simple. This is called the \emph{Hasse invariant}.

We denote by $\overline{\rm Hdg}$ the locally principal ideal of $\cO_{\bar{X}_{\mathrm{int}}}$ described as follows: for each $U=\Spec(R)\subset\bar{X}_{\mathrm{int}}$  if $\omega_{0}\mid_U=Rw$ and  $\mathrm{Ha}\mid_U=Hw^{\otimes(p-1)}$ then $\overline{\rm Hdg}\mid_U=HR\subseteq R$. Let $\rm Hdg$ the inverse image of $\overline{\rm Hdg}$ in $\cO_{\dX}$, which is also a locally principal ideal. Note that $\mathrm{Ha}^{p^n}$ extends canonically to a section of $H^0(\dX,\omega_0^{p^n(p-1)}\otimes\Z/p^{n+1}\Z)$, indeed, for any two extensions $\mathrm{Ha}_1$ and $\mathrm{Ha}_2$ of $\mathrm{Ha}$ we have $\mathrm{Ha}_1^{p^n}=\mathrm{Ha}_2^{p^n}$ modulo $p^{n+1}$ by the binomial formula. 

\begin{rmk} From \cite[Prop. 3.4]{brasca13} there exists a $(p-1)$-root of the principal ideal $\mathrm{Hdg}$. This ideal is denoted $\underline{\delta}$ or $\mathrm{Hdg}^{1/(p-1)}$.
\end{rmk}

We introduce the corresponding neighborhoods of the ordinary locus. 
For each integer $r\geq 1$ we denote by $\dX_{r}$ the formal scheme over $\dX$ which represents the functor that classifies for each $p$-adically complete  $\hat\cO$-algebra $R$:
$$\dX_{r}(R)= \left\lbrace [(h, \eta)] \ \ | \ \  h\in\dX(R),\; \eta\in H^0(\Spf(R),h^\ast(\omega_{\cG}^{(1-p)p^{r+1}})),\; \eta\cdot{\rm Ha}^{p^{r+1}}=p\mod p^2 \right\rbrace ,$$
here the brackets means the equivalence class given by $(h,\eta)\equiv (h',\eta')$ if $h=h'$ and $\eta=\eta'(1+pu)$ for some $u\in R$ 
(see \cite[Definition 3.1]{AIP-halo}).

\begin{prop}\cite[Corollaire A.2]{AIP-halo}\label{p:canonical subgroups} There exists a canonical subgroup $C_n$ of $\cG_0[\mathfrak{p}_0^n]$ for $n\leq r$ over $\dX_{r}$. This is unique and satisfy the compatibility $C_n[\mathfrak{p}_0^{n-1}]=C_{n-1}$. Moreover, if we denote $D_n:= \cG_0[\mathfrak{p}_0^n]/C_n$ then $\omega_{D_n}\simeq\omega_{\cG_0[\dP^n]}/\mathrm{Hdg}^{\frac{p^n-1}{p-1}}$.
\end{prop}

If we write $\Omega_0\subseteq\omega_0$ for the subsheaf generated by the lifts of the image of the Hodge-Tate map, by \cite[Proposition A.3]{AIP-halo} we have that $\Omega_0=\underline{\delta}\omega_0$. Thus, we obtain a morphism
\begin{equation}\label{dlog0}
{\rm dlog}_0:D_n(\dX_{r})\longrightarrow \Omega_0\otimes_{\cO_{\dX_{r}}}(\cO_{\dX_{r}}/\cI_n)\subset \omega_0/p^n\mathrm{Hdg}^{-\frac{p^n-1}{p-1}},
\end{equation}
where $\cI_n:=p^n\mathrm{Hdg}^{-\frac{p^n}{p-1}}$.

By the moduli interpretation, the $p$-divisible group $\prod_{\dP\neq\dP_0}\cG_\dP\rightarrow\dX_{r}$ is \'etale isomorphic to $\prod_{\dP\neq\dP_0}(F_\dP/\cO_\dP)^2$. Assume that $r\geq n$. We denote by $\dX_{r,n}\rightarrow\dX_{r}$ the formal scheme obtained by adding to the moduli interpretation a point of order $\mathfrak{p}^{n}$ for each $\mathfrak{p} \neq \mathfrak{p}_0$ whose multiples generate $\cG_\dP[\dP]/C_\dP$ (see Remark \ref{rmkonpts}). It is clear that the extension $\dX_{r,n}\rightarrow\dX_{r}$ is \'etale and its Galois group contains $\prod_{\dP\neq\dP_0}(\cO_\dP/p^{n}\cO_\dP)^\times$ as a subgroup. Moreover, $\dX_{r,n}$ has also good reduction (see \cite[\S 5.4]{carayol86}). Now we trivialize the subgroup $D_n$:
Let  $\cX_{r,n}$ be the adic generic fiber of $\dX_{r,n}$. By \cite[Corollaire A.2]{AIP-halo}, the group scheme $D_{n}\rightarrow\cX_{r,n}$ is also \'etale isomorphic to $p^{{-n}}\cO_{\dP_0}/\cO_{\dP_0}$. We denote by $\mathcal{IG}_{r,n}$ the adic space over $\cX_{r,n}$ that trivialize $D_{n}$. Then the map $\mathcal{IG}_{r,n} \rightarrow\cX_{r,n}$ is a finite \'etale with Galois group $(\cO_{\dP_0}/p^{n}\cO_{\dP_0})^\times$. We denote by $\mathfrak{IG}_{r,n}$ the normalization $\mathcal{IG}_{ r,n}$ in $\dX_{ r,n}$ which is finite over $ \dX_{ r,n}$ and it is also endowed with an action of $(\cO_{{\dP_0}}/p^{n}\cO_{{\dP_0}})^\times$. These constructions are captured by the following tower of formal schemes: 
$$\mathfrak{IG}_{r,n} \longrightarrow\dX_{r,n}\longrightarrow\dX_{r},$$
endowed with a natural action of $(\cO/p^{n}\cO)^\times\simeq\prod_\dP(\cO_\dP/p^{n}\cO_\dP)^\times$.

Let us consider the morphism
\[
\eta:\mathfrak{IG}_{r,n} \longrightarrow\dX_{r,n}\longrightarrow\dX_{r}\longrightarrow\dX.
\]
The following result is completely analogous to \cite[Lemma 3.3]{AI17}:
\begin{lemma}\label{lemOmegaOmega}
The induced map $\eta^*\left(\Omega_{\dX/\hat\cO}^1\right)\longrightarrow\Omega_{\mathfrak{IG}_{r,n}/\hat\cO}^1$ has kernel and cokernel annihilated by a power of $\underline{\delta}$ and, in particular, by a power of $p$, depending on $n$.
\end{lemma}

\subsection{Formal vector bundles}\label{ss:formal vector bundles}
	We briefly recall in this subsection constructions performed in \cite[\S 2]{AI17} and \cite[\S 6]{BM}. Let $S$ be a formal scheme, $\cI$ its (invertible) ideal of definition and $\mathcal{E}$ a locally free $\cO_S$-module of rank $n$. We write $\bar S$ the scheme with structural sheaf $\mathcal{O}_S/\cI$ and put $\overline{\mathcal{E}}$ the corresponding $\cO_{\overline{S}}$-module. We fix marked sections $s_1,\cdots, s_m$ of $\overline{\mathcal{E}}$, namely, the sections $s_1, \cdots , s_m$ define a direct sum decomposition $\overline{\mathcal{E}} = \cO_{\bar S}^m\oplus Q$, where $Q$ is a
locally free $\cO_{\bar S}$-module of rank $n-m$.

Let $S-\mathrm{Sch}$ be the category of the  formal $S$-schemes. There exists a formal scheme $\V(\mathcal{E})$ over $S$ called the \emph{formal vector bundle} attached to $\mathcal{E}$ which represents the functor, denoted by the same symbol, $S-\mathrm{Sch} \rightarrow \mathrm{Sets}$, given by $\V(\mathcal{E})(t:T\rightarrow S):=  H^0(T,t^\ast(\mathcal{E})^\vee)=  {\rm Hom}_{\cO_T}(t^\ast(\mathcal{E}),\cO_T)$.  Crucial in \cite{AI17} is the construction of the so called \emph{formal vector bundles with marked sections} which is the formal scheme $\V_0(\mathcal{E},s_1,\cdots,s_m)$ over $\V(\mathcal{E})$ that represents the sub-functor $S-\mathrm{Sch} \rightarrow \mathrm{Sets}$
   $$ \V_0(\mathcal{E}, s_1,\cdots,s_m)(t:T\rightarrow S)=\left\lbrace \rho\in H^0(T,t^\ast(\mathcal{E})^\vee) \ | \  \bar\rho(t^\ast(s_i))=1, \  i=1,\cdots,m \right\rbrace, $$
here $\bar\rho$ is the reduction of $\rho$ modulo $\cI$. 

Given the fixed decomposition $\overline{\mathcal{E}}=Q\oplus \langle s_i\rangle_i$,
let us consider now the sub-functor $\V_Q(\mathcal{E},s_1,\cdots,s_m)$ that associates to any formal $S$-scheme $t:T\rightarrow S$ the subset of sections $\rho\in \V_0(\mathcal{E},s_1,\cdots,s_m)(T)$ whose reduction $\bar\rho$ modulo $\cI$ also satisfies $\bar\rho(t^\ast(m))=0$ for every $m\in Q$.
\begin{lemma}\cite[Lemma 6.3]{BM} 
The morphism $\V_{Q}(\mathcal{E},s_1,\cdots,s_m)\rightarrow \V_0(\mathcal{E},s_1,\cdots,s_m)$ is represented by a formal subscheme.
\end{lemma}

\begin{rmk}
Notice that this construction is also functorial with respect to $(\mathcal{E},Q,s_1,\cdots,s_m)$. Indeed, given a morphism $\varphi:\mathcal{E}'\rightarrow\mathcal{E}$ of locally free $\cO_S$-modules of finite rank and marked sections $s_1,\cdots,s_m\in \overline{\mathcal{E}}$, $s_1',\cdots,s_m'\in \overline{\mathcal{E}'}$ such that $\bar\varphi(s_i')=s_i$ and $\bar\varphi(Q')\subseteq Q$, we have the morphisms making the following diagram commutative
\[
\xymatrix{
\V_{Q}(\mathcal{E},s_1,\cdots,s_m)\ar[r]\ar[d]&\V_{0}(\mathcal{E},s_1,\cdots,s_m)\ar[r]\ar[d]&\V(\mathcal{E})\ar[d]\\
\V_{Q}(\mathcal{E}',s_1',\cdots,s_m')\ar[r]&\V_{0}(\mathcal{E}',s_1',\cdots,s_m')\ar[r]&\V(\mathcal{E}')
}
\]
\end{rmk}
\begin{rmk}\label{DepOnQuot}
In fact, $\V_{Q}(\mathcal{E},s_1,\cdots,s_m)$ depends on $Q\subset \overline{\mathcal{E}}$ and the image of $s_i$ in $\overline{\mathcal{E}}/Q$. Indeed, given $m_i\in Q$ and $t:T\rightarrow S$, since $\bar\rho(t^\ast(m_i))=0$,
\[
\V_{Q}(\mathcal{E},(s_i+m_i)_i)(T)=\left\{\rho\in H^0(T,t^\ast(\mathcal{E})^\vee);\quad\bar\rho(t^\ast(s_i+m_i))=1,\; \bar\rho(t^\ast(Q))=0\right\}=\V_{Q}(\mathcal{E},(s_i)_i)(T).
\]
\end{rmk}

\subsubsection{Filtrations}\label{filFVB} 
Let $\mathcal{E}$ be a locally free $\cO_S$-module of rank $h$ and assume that there exists an $\cO_S$-submodule $\cF\subset\mathcal{E}$, locally free as $\cO_S$-module of rank $d$, which is a locally direct summand in $\mathcal{E}$. Assume also that we have marked sections $s_1,\cdots,s_d$ of $\overline{\mathcal{E}}$ that define a $\cO_{\bar S}$-basis of $\overline{\cF}$. Assume as above that we fix a direct summand $\overline{\mathcal{E}}=Q\oplus\langle s_i\rangle=Q\oplus\overline{\mathcal{F}}$. The following commutative diagram is obtained by functionality
\[
\xymatrix{
\V_Q(\mathcal{E},s_1,\cdots,s_d)\ar[r]^{g_Q}\ar[rd]&\V_0(\mathcal{E},s_1,\cdots,s_d)\ar[d]\ar[r]&\V(\mathcal{E})\ar[d]\\
&\V_0(\cF,s_1,\cdots,s_d)\ar[r]&\V(\cF)
}
\]
By \cite[Lemma 2.5]{AI17}, the diagram on the right hand side is cartesian and the vertical morphisms are principal homogeneous spaces under the group of affine transformations $\A_S^{h-d}$.
As a corollary, if $f:\V(\mathcal{E})\rightarrow S$ and $f_0:\V_0(\mathcal{E},s_1\cdots,s_d)\rightarrow S$ are the structural morphisms, $f_{0,*}\cO_{\V_0(\mathcal{E},s_1,\cdots,s_d)}$ is endowed with increasing filtrations ${\rm Fil}_\bullet f_{0,*}\cO_{\V_0(\mathcal{E},s_1,\cdots,s_d)}$ with graded pieces
\begin{equation}\label{gradpiec}
{\rm Gr}_nf_{0,*}\cO_{\V_0(\mathcal{E},s_1,\cdots,s_d)}\simeq f_{0,*}\cO_{\V_0(\mathcal{F},s_1,\cdots,s_d)}\otimes_{\cO_S}{\rm Sym}^n(\mathcal{E}/\mathcal{F}).
\end{equation}
If we consider the structural morphism $f_Q=f_0\circ g_Q:\V_Q(\mathcal{E},s_1\cdots,s_d)\rightarrow S$,
this defines a filtration ${\rm Fil}_\bullet f_{Q,*}\cO_{\V_Q(\mathcal{E},s_1,\cdots,s_d)}$ in $f_{Q,*}\cO_{\V_Q(\mathcal{E},s_1,\cdots,s_d)}$,
\[
{\rm Fil}_n f_{Q,*}\cO_{\V_Q(\mathcal{E},s_1,\cdots,s_d)}:=f_{Q,*}\cO_{\V_Q(\mathcal{E},s_1,\cdots,s_d)}\cap {\rm Fil}_n f_{0,*}\cO_{\V_0(\mathcal{E},s_1,\cdots,s_d)}.
\]
The filtration is characterized by its local description: Over $U={\rm Spf}(R)\subset S$, an open affine subscheme such that $\mathcal{F}$, $\mathcal{E}$ are free over $U$, we have that 
\[
\V_0(\mathcal{F},s_1,\cdots,s_d)\mid_U={\rm Spf}(R\langle Z_1,\cdots,Z_d\rangle),\quad \V_0(\mathcal{E},s_1,\cdots,s_d)\mid_U={\rm Spf}(R\langle Z_1,\cdots,Z_d,Y_1,\cdots,Y_{h-d}\rangle).
\]
Then ${\rm Fil}_n f_{0,*}\cO_{\V_0(\mathcal{E},s_1,\cdots,s_d)}(U)$ consists of the polynomials of degree $\leq n$ in the variables $Y_1,\cdots,Y_{h-d}$ with coefficients in $R\langle Z_1,\cdots,Z_d\rangle$. 
Similarly, if $\alpha\in\cI$ is a generator, we have that $\V_Q(\mathcal{E},s_1,\cdots,s_d)={\rm Spf}(R\langle Z_1,\cdots,Z_d,T_1,\cdots,T_{h-d}\rangle)$ and
\[
g_Q^\ast:R\langle Z_1,\cdots,Z_d,Y_1,\cdots,Y_{h-d}\rangle\rightarrow R\langle Z_1,\cdots,Z_d,T_1,\cdots,T_{h-d}\rangle, \qquad g_Q^\ast(Y_i)=\alpha T_i.
\] 
This implies that ${\rm Fil}_n f_{Q,*}\cO_{\V_Q(\mathcal{E},s_1,\cdots,s_d)}(U)$ consists of the polynomials of degree $\leq n$ in the variables $T_1,\cdots,T_{h-d}$ with coefficients in $R\langle Z_1,\cdots,Z_d\rangle$, and $g_Q^\ast$ provides an isomorphism 
\begin{equation}\label{isoFils}
{\rm Fil}_n f_{0,*}\cO_{\V_0(\mathcal{E},s_1,\cdots,s_d)}(U)[\alpha^{-1}]\simeq {\rm Fil}_n f_{Q,*}\cO_{\V_Q(\mathcal{E},s_1,\cdots,s_d)}(U)[\alpha^{-1}].
\end{equation}
Finally, both constructions of the filtrations are functorial in sense of \cite[Corollary 2.7]{AI17}.

\subsubsection{Connections}\label{connFVB} The following part is a brief summary of \cite[\S 2.4]{AI17}.
Suppose that we have a $\Z_p$-algebra $A_0$ and an element $\varpi\in A_0$ such that $A_0$ is $\varpi$-adically complete and separated. Let $S$ be a formal scheme locally of finite type over ${\rm Spf}(A_0)$ such that the topology of $S$ is the $\varpi$-adic topology. We let $\Omega_{S/A_0}^1$ be the $\cO_S$-module of continuous Kh\"aler differentials.

Consider a free $\cO_S$-module $\mathcal{E}$ endowed with an integrable connection $\bigtriangledown:\mathcal{E}\longrightarrow \mathcal{E}\otimes_{\cO_S}\Omega_{S/A_0}^1$. Assume that we have fixed marked sections $s_1,\cdots,s_d\in \overline{\mathcal{E}}$ which are horizontal for the reduction of $\bigtriangledown$ modulo $\cI=\varpi\cO_S$. Then by \cite[\S 2.4]{AI17} $\bigtriangledown$ defines an integrable connection
\[
\bigtriangledown_0:\V_0(\mathcal{E},s_1,\cdots,s_d)\longrightarrow \V_0(\mathcal{E},s_1,\cdots,s_d)\hat\otimes_{\cO_S}\Omega^1_{S/A_0}
\]
where $\hat\otimes$ denotes the completed tensor product.

Moreover, if we assume that we are in the situation of the previous section with a locally free $\cO_S$-module and a direct summand
$\mathcal{F}\subset\mathcal{E}$, we consider the filtrations ${\rm Fil}_\bullet f_{0,*}\cO_{\V_0(\mathcal{E},s_1,\cdots,s_d)}$. By \cite[Lemma 2.9]{AI17} the connection $\bigtriangledown_0$ satisfies Griffith's transversality with respect to the filtration ${\rm Fil}_\bullet f_{0,*}\cO_{\V_0(\mathcal{E},s_1,\cdots,s_d)}$, namely for every integer $n$ we have
\[
\bigtriangledown_0\left({\rm Fil}_n f_{0,*}\cO_{\V_0(\mathcal{E},s_1,\cdots,s_d)}\right)\subset {\rm Fil}_{n+1} f_{0,*}\cO_{\V_0(\mathcal{E},s_1,\cdots,s_d)}\hat\otimes_{\cO_S}\Omega_{S/A_0}^1.
\]
Furthermore the induced map
\begin{equation}\label{gradpiecD}
{\rm gr}_n(\bigtriangledown_0):{\rm Gr}_nf_{0,*}\cO_{\V_0(\mathcal{E},s_1,\cdots,s_d)}\longrightarrow {\rm Gr}_{n+1}f_{0,*}\cO_{\V_0(\mathcal{E},s_1,\cdots,s_d)}\hat\otimes_{\cO_S}\Omega_{S/A_0}^1
\end{equation}
is $\cO_S$-linear and, via the identification \eqref{gradpiec}, 
the morphism ${\rm gr}_\bullet(\bigtriangledown_0)$ is ${\rm Sym}^\bullet(\mathcal{E}/\mathcal{F})$-linear.

\subsection{Weight space}\label{ss:weight space}   We fix a decomposition: 
$$\cO^\times\cong \cO^0\times H,$$
where $H$ is the torsion subgroup of $\cO^\times$ and $\cO^0\simeq 1+p\cO$ is a free $\Z_p$-module of rank $d$.
We put $\Lambda_F:= \Z_p[[\cO^\times]]$ and $\Lambda_F^0:= \Z_p[[\cO^0]]$. The choice of a basis $\{e_1, ..., e_d\}$ of  $\cO^0$ furnishes an isomorphism $\Lambda_F^0\cong  \Z_p[[T_1, ....,T_d]]$ given by  $1+ T_i= e_i$ for $i= 1,..., d$. Moreover, for each $n\in \N$ we consider the algebras:
\[
\Lambda_{n}:= \Lambda_F\left\langle\frac{T_1^{p^{n-1}}}{p},\cdots,\frac{T_d^{p^{n-1}}}{p}\right\rangle\qquad\qquad \Lambda_{n}^0:=\Lambda_F^0\left\langle\frac{T_1^{p^{n-1}}}{p},\cdots,\frac{T_d^{p^{n-1}}}{p}\right\rangle
\]

The formal scheme $\dW=\Spf(\Lambda_F)$ is our formal \emph{weight space} for $G'$.  Thus for each complete $\Z_p$-algebra $R$ we have: 
$$\dW(R)={\rm Hom}_{\mathrm{cont}}(\cO^\times,R^\times).$$
We also consider the following formal schemes $\dW^0=\Spf(\Lambda_F^0)$, $\dW_{n}:=\Spf(\Lambda_{n})$ and $\dW^0_{n}:=\Spf(\Lambda_{n}^0)$. By construction we have $\dW=\bigcup_n\dW_{n}$ and $\dW^0=\bigcup_n\dW^0_{n}$. Moreover, we have the following explicit description:
\[
\dW_n^0(\C_p)=\{k\in{\rm Hom}_{\mathrm{cont}}(\cO^0,\C_p^\times): \;|k(e_i)-1|\leq p^{-p^{-n+1}},\;i=1,\cdots,d\}
\] 
We denote by ${\bf k}:\cO^\times\rightarrow\Lambda_F^\times$ the universal character of $\dW$, which decomposes as ${\bf k}= {\bf k}^0\otimes{\bf k}_f$ where: 
 \[
 {\bf k}_f:H\longrightarrow \Z_p[H]^\times\qquad\qquad {\bf k}^0:\cO^0\longrightarrow (\Lambda_F^0)^\times.
 \]
Let ${\bf k}_{n}^0: \cO^0 \rightarrow (\Lambda^0_n)^\times $ be  given by the composition of ${\bf k}^0$ with the inclusion $(\Lambda_F^0)^\times\subseteq (\Lambda^0_n)^\times$ and we put ${\bf k}_n:={\bf k}_n^0\otimes{\bf k}_f:\cO^\times\rightarrow\Lambda_n^\times$. The following Lemma can be found in \cite[Lemma 6.4]{BM}, but it also can be deduced from the computations below.

 \begin{lemma}\label{weightextends} Let $R$ be a $p$-adically complete $\Lambda^0_{n}$-algebra. Then ${\bf k}_{n}^0$ extends locally analytically to a character $\cO^0(1+p^{n}\lambda^{-1}\cO_{F}\otimes R)\rightarrow R^{\times}$, for any $\lambda\in R$ such that $\lambda^{p-1}\in p^{p-2}m_R$, where $m_R$ is the maximal of $R$. In particular, ${\bf k}_{n}^0$ is analytic on $1+p^{n}\cO$.
\end{lemma}	

 Recall that hypothesis \ref{hypothesis 1} imply a decomposition of rings $\cO=\cO_{\dP_0} \times\prod_{\mathfrak{p}\neq \mathfrak{p}_0}\cO_{\mathfrak{p}}= \Z_p\times\cO^{\tau_0}$ and put $\cO^{\tau_0,0}:=1+p\cO^{\tau_0}$. Analogously as above we introduce:
 \begin{eqnarray*}
 \begin{gathered}
 \Lambda_{\tau_0}:=\Z_p[[\Z_p^\times]]\qquad\qquad \Lambda_{\tau_0}^0:=\Z_p[[1+p\Z_p]]\simeq\Z_p[[T]]\\
\Lambda^{\tau_0}:=\Z_p[[\cO^{\tau_0\times}]]\qquad\qquad  \Lambda^{\tau_0,0}:=\Z_p[[\cO^{\tau_0,0}]]\simeq\Z_p[[T_2,\cdots,T_d]]\\
\Lambda_{\tau_0,n}:= \Lambda_{\tau_0}\left\langle\frac{T^{p^{n-1}}}{p}\right\rangle\qquad\qquad \Lambda_{\tau_0,n}^0:= \Lambda_{\tau_0}^0\left\langle\frac{T^{p^{n-1}}}{p}\right\rangle\\
 \Lambda^{\tau_0}_n:=\Lambda^{\tau_0}\left\langle\frac{T_2^{p^{n-1}}}{p},\cdots,\frac{T_d^{p^{n-1}}}{p}\right\rangle\qquad\qquad \Lambda^{\tau_0,0}_n:=\Lambda^{\tau_0,0}\left\langle\frac{T_2^{p^{n-1}}}{p},\cdots,\frac{T_d^{p^{n-1}}}{p}\right\rangle.
 \end{gathered}
\end{eqnarray*}
Thus, we have decompositions $\Lambda_n^0=\Lambda_{\tau_0,n}^0\hat\otimes\Lambda_n^{\tau_0,0}$ and $\Lambda_n=\Lambda_{\tau_0, n}\hat\otimes\Lambda_n^{\tau_0}$. We denote by ${\bf k}_{\tau_0, n}^{0}:(1+p\Z_p)\longrightarrow\Lambda_{\tau_0,n}^0$, ${\bf k}_{n}^{\tau_0,0}:\cO^{\tau_0,0}\longrightarrow\Lambda_n^{\tau_0,0}$, ${\bf k}_{\tau_0, n}:\Z_p^\times\rightarrow \Lambda_{\tau_0,n}$ and ${\bf k}_n^{\tau_0}:\cO^{\tau_0\times}\rightarrow \Lambda^{\tau_0}_n$ the universal characters. Then we have $ {\bf k}_{n}^0= {\bf k}_{\tau_0, n}^0\otimes {\bf k}_{n}^{\tau_0,0}$ and ${\bf k}_n={\bf k}_{\tau_0, n}\otimes {\bf k}_n^{\tau_0}$. Moreover,
 $${\bf k}_{\tau_0, n}\otimes {\bf k}_n^{\tau_0}= {\bf k}_{n}= {\bf k}_{n}^0\otimes{\bf k}_{f}={\bf k}_{\tau_0, n}^0\otimes {\bf k}_{n}^{\tau_0,0}\otimes{\bf k}_{f}.$$

\subsubsection{The universal character ${\bf k}_{\tau_0,n}^0$}\label{expdescunivchar} Recall that the weight space $\dW_{\tau_0,n}^0={\rm Spf}(\Lambda_{\tau_0,n}^0)$ classifies characters such that
\begin{equation}\label{Wn}
\dW_{\tau_0,n}^0(\C_p)=\left\{k\in{\rm Hom}_{\mathrm{cont}}(1+p\Z_p,\C_p^\times): \;|k(\exp(p))-1|\leq p^{-p^{-n+1}}\right\}.
\end{equation}
In this part we will describe the universal character ${\bf k}_{\tau_0,n}^0$. By the above lemma ${\bf k}_{\tau_0,n}^0$ is analytic when restricted to $1+p^n\Z_p$. Moreover, it is given by 
\[
1+p\Z_p\stackrel{{\bf k}_{\tau_0}^0}{\longrightarrow}\Lambda_{\tau_0}^\times\hookrightarrow \Lambda_{\tau_0,n}^\times,\qquad  {\bf k}_{\tau_0}^0(\exp(\alpha p))=(1+T)^\alpha,\quad\alpha\in \Z_p.
\] 
Notice that $(1+T)^\alpha=\exp(\alpha u_n)$, where $u_n:=\log(1+T)\in p^{-n+2}\Lambda_{\tau_0,n}^0$ since 
\[
\log(1+T)=-\sum_{k\geq 1}\frac{(-T)^k}{k}=p^{-n+2}\sum_{k\in p^{n-1}\N}(-1)^{k+1}\frac{T^{k-(v(k)-n+2)p^{n-1}}}{kp^{-v(k)}}\left(\frac{T^{p^{n-1}}}{p}\right)^{v(k)-n+2}-\sum_{k\not\in p^{n-1}\N}\frac{(-T)^k}{k},
\]
where $v:\Z\rightarrow\N$ is the $p$-adic valuation. This implies that
\begin{eqnarray*}
(1+T)^\alpha&=&\sum_{j=0}^{p^{n-1}-1}(1+T)^j\cdot 1_{j+p^{n-1}\Z_p}(\alpha)\cdot\exp((\alpha-j)u_n)\\
&=&\sum_{j=0}^{p^{n-1}-1}(1+T)^j\cdot 1_{j+p^{n-1}\Z_p}(\alpha)\cdot\sum_{i\geq 0}\binom{p^{-1}u_n}{i}\left(\frac{\exp(\alpha p)-\exp(jp)}{\exp(jp)}\right)^i,
\end{eqnarray*}
where the last equality comes from the following equality in $\Q[[X,Y]]$
\[
\exp(X\log(1+Y))=\sum_{i\geq 0}\binom{X}{i}Y^i, \qquad \binom{X}{i}:=\frac{X(X-1)\cdots(X-i+1)}{i!}.
\]
\begin{rmk}\label{valbinom}
Notice that $v(h!)\leq\frac{h}{p-1}$ (see \cite[Lemma 4.12]{AI17}), hence we have 
\[
\binom{p^{-1}u_n}{i}\in p^{-i\left(n-\frac{p}{p-1}\right)}\Lambda_{\tau_0,n}^0.
\]
This implies that $\sum_{i\geq 0}\binom{p^{-1}u_n}{i}\left(\frac{\exp(\alpha p)-\exp(jp)}{\exp(jp)}\right)^i$ converges since $\left(\frac{\exp(\alpha p)-\exp(jp)}{\exp(jp)}\right)^i\in p^{in}\Z_p$.
\end{rmk}
Thus, we obtain that the locally analytic expression of ${\bf k}_{\tau_0,n}^0$ is given by
\begin{equation}\label{univK}
{\bf k}_{\tau_0,n}^0(\beta)=\sum_{j=0}^{p^{n-1}-1}(1+T)^j\cdot 1_{\exp(jp)+p^{n}\Z_p}(\beta)\cdot\sum_{i\geq 0}\binom{p^{-1}u_n}{i}\left(\frac{\beta-\exp(jp)}{\exp(jp)}\right)^i.
\end{equation}
The following examples help us to visualize that the above expression can be specialized at both analytic character and locally constant characters that factor through $(1+p\Z_p)/(1+p^{n+1}\Z_p)$.
\begin{ex}
Let $k\in \dW_{\tau_0,n}^0(\C_p)$ be the point corresponding to the character $\beta\mapsto \beta^k$, for some $k\in\Z$. It clearly satisfies the characterization of Equation \eqref{Wn}. It corresponds to the maximal ideal $(T-\exp(kp)+1)$. We check that 
\begin{eqnarray*}
k^\ast{\bf k}_{\tau_0,n}^0(\beta)&=&\sum_{j=0}^{p^{n-1}-1}\exp(pk)^j\cdot 1_{\exp(jp)+p^n\Z_p}(\beta)\cdot\sum_{i= 0}^k\binom{k}{i}\left(\frac{\beta-\exp(jp)}{\exp(jp)}\right)^i\\
&=&\sum_{j=0}^{p^{n-1}-1}\exp(jp)^k\cdot 1_{\exp(jp)+p^{n}\Z_p}(\beta)\left(\frac{\beta}{\exp(jp)}\right)^k=\beta^k.
\end{eqnarray*}
\end{ex}

\begin{ex}
Let $\chi\in \dW_{\tau_0,n}^0(\C_p)$ be the point corresponding to the character $\chi(\exp(\alpha p))=\xi_n^\alpha$, for a fix $(n-1)$-th root of unity $\xi_n$. It also satisfies the characterization of Equation \eqref{Wn} and corresponds to the maximal ideal $(T-\xi_n+1)$. We check that 
\begin{eqnarray*}
\chi^\ast{\bf k}_{\tau_0,n}^0(\exp(\alpha p))&=&\sum_{j=0}^{p^{n-1}-1}\xi_n^j\cdot 1_{\exp(jp)+p^{n}\Z_p}(\exp(\alpha p))\cdot\sum_{i\geq 0}\binom{0}{i}\left(\frac{\exp(\alpha p)-\exp(jp)}{\exp(jp)}\right)^i\\
&=&\sum_{j=0}^{p^{n-1}-1}\xi_n^j\cdot 1_{j+p^{n-1}\Z_p}(\alpha)=\xi_n^\alpha.
\end{eqnarray*}
\end{ex}

\subsection{Overconvergent modular sheaves}\label{ss:overconvergent modular sheaves} 

We fix $L$  a finite extension of $\Q_p$ containing all the $p$-adic embedding of $F$ and let us work over the ring of integers of $L$. Let $r\geq n$. As  in \S \ref{ss:integral models and divisible groups}, we consider the ideal of $\cO_{\mathfrak{IG}_{ r,n}}$ given by $\mathcal{I}_n:=p^{n} \mathrm{Hdg}^{-\frac{p^{n}}{p-1}}$, which is our ideal in order to perform the construction of \S\ref{ss:formal vector bundles}. Then using notations from \S\ref{ss:formal vector bundles} we put $\overline{\mathfrak{IG}_{r,n}}$ for the corresponding reduction modulo $\cI_{n}$. From Equation \eqref{dlog0}, we have an isomorphism 
\begin{equation}\label{omega canonical vs G}
\mathrm{dlog}_0:D_{n}(\mathfrak{IG}_{r,n})\otimes_{\Z_p}(\cO_{\mathfrak{IG}_{r,n}}/\cI_n)\stackrel{\simeq}{\longrightarrow}\Omega_0\otimes_{\cO_{\mathfrak{IG}_{ r,n}}}(\cO_{\mathfrak{IG}_{r,n}}/\cI_n),
\end{equation}
where $\Omega_0$ is the $\cO_{\mathfrak{IG}_{ r,n}}$-submodule of $\omega_0$ generated by all the lifts of $\mathrm{dlog}_0(D_{n})$. 
By construction, there exist on $\mathfrak{IG}_{r,n}$ a universal canonical generator $P_{0,n}$ of $D_{n}$, and universal points $P_{\dP,n}$ of order $p^{n}$ in $\cG_\dP[p^{n}]$. We put:
\begin{equation}\label{e:omega for the machinery}
\Omega:= \Omega_0\oplus \bigoplus_{\mathfrak{p} \neq \mathfrak{p}_0}\omega_{\mathfrak{p}}= \Omega_0\oplus\Omega^0
\end{equation}
 where $\Omega^0= \bigoplus_{\mathfrak{p} \neq \mathfrak{p}_0}\omega_{\mathfrak{p}}$, and we denote $\overline{\Omega}$ the associated $\cO_{\overline{\mathfrak{IG}_{r,n}}}$-module. Now we produce marked sections in $\overline{\Omega}$ as follows. Let $\mathfrak{p}\mid p$ and we consider two cases: 
\begin{itemize}
	\item if $\mathfrak{p}= \mathfrak{p}_0$ we denote by $s_0 \in \overline{\Omega}$ the image $\mathrm{dlog}_0(P_{0,n})$ using the isomorphism (\ref{omega canonical vs G}). 
	\item if $\mathfrak{p} \neq \mathfrak{p}_0$ we firstly consider the decomposition $\omega_{\mathfrak{p}}= \oplus_{\tau \in \Sigma_\mathfrak{p}}\omega_{\mathfrak{p}, \tau}$ over $\mathfrak{IG}_{r,n}$ and the dlog map: 
	\begin{equation}\label{dlogP}
	\mathrm{dlog}_{\mathfrak{p}}: \cG_{\mathfrak{p}}[p^{n}] \longrightarrow \omega_{\cG_{\mathfrak{p}}[p^{n}]^D}= \omega_{\mathfrak{p}}/p^{n}\omega_{\mathfrak{p}}= \bigoplus_{\tau \in \Sigma_\mathfrak{p}}\omega_{\mathfrak{p}, \tau}/p^{n}\omega_{\mathfrak{p}, \tau}.
	\end{equation}
Hence the image of $P_{\mathfrak{p}, n}$ through $\mathrm{dlog}_{\mathfrak{p}}$ provides a set of sections $\{s_{\mathfrak{p}, \tau}\}_{\tau \in \Sigma_\mathfrak{p}}$ of $\omega_{\mathfrak{p}, \tau}/\cI_n\omega_{\mathfrak{p}, \tau}\subseteq\omega_{\mathfrak{p}}/\cI_{n}\omega_{\mathfrak{p}}$. 
\end{itemize}
By \cite[Lemma 6.5]{BM}, the set ${\bf s}:=\{s_0\}\cup\bigcup_{\mathfrak{p}\neq \mathfrak{p}_0} \{s_{\mathfrak{p}, \tau}\}_{\tau \in \Sigma_\mathfrak{p}}$ define marked sections for the locally free $\cO_{\mathfrak{IG}_{r,n}}$-module $\Omega$. Hence we construct the formal scheme $\V_0(\Omega,{\bf s})$ over $\mathfrak{IG}_{r,n}$. By construction we have the following tower of formal schemes:  
$$\xymatrix{\V_0(\Omega,{\bf s})\ar[r]&\mathfrak{IG}_{r,n} \ar^{g_n}[r]&\dX_{r}.}$$
For any $\dX_r$-scheme $T$ we have:
\[
\V_0(\Omega,{\bf s})(T)=\{(\rho,\varphi)\in \mathfrak{IG}_{r,n}(T)\times\Gamma(T,\rho^\ast\Omega^\vee);\quad  \varphi(\rho^\ast s_i)\equiv1\mod \cI_n\}
\]

Let ${\bf s}^0:=\bigcup_{\mathfrak{p}\neq \mathfrak{p}_0} \{s_{\mathfrak{p}, \tau}\}_{\tau \in \Sigma_\mathfrak{p}}$ then from (\ref{e:omega for the machinery}) we have:
\[
\V_0(\Omega,{\bf s})=\V_0(\Omega_0,s_0)\times_{\mathfrak{IG}_{r,n}}\V_0(\Omega^0,{\bf s}^0).
\]

As we are interested in locally analytic distributions (rather than functions) we perform the following construction. Let $t_{\mathfrak{p}, \tau}\in\bar\omega_{\dP,\tau}$ be any section such that $\varphi_\tau(s_{\mathfrak{p}, \tau}\wedge t_{\mathfrak{p}, \tau})=1$ and $Q_{{\bf s}^0}\subset \overline{\Omega^0}$ be the direct summand generated by the sections in ${\bf s}^0$. We put ${\bf t}^0:=\bigcup_{\mathfrak{p}\neq \mathfrak{p}_0} \{t_{\mathfrak{p}, \tau}\}_{\tau \in \Sigma_\mathfrak{p}}$ and 
\[
\V_{Q_{{\bf s}^0}}(\Omega^0,{\bf t}^0)\stackrel{f^0}{\longrightarrow}\mathfrak{IG}_{r,n}\stackrel{g_n}{\longrightarrow}\dX_{r}, \qquad \V_{0}(\Omega_0,s_0)\stackrel{f_0}{\longrightarrow}\mathfrak{IG}_{r,n}\stackrel{g_n}{\longrightarrow}\dX_{r}.
\]
The sections $t_{\mathfrak{p}, \tau}$ are well defined modulo $Q_{{\bf s}^0}$ which is fine because remark \ref{DepOnQuot}. We have 
\[
\V_{Q_{{\bf s}^0}}(\Omega^0,{\bf t}^0)(T)=\{(\rho,\varphi)\in \mathfrak{IG}_{r,n}(T)\times\Gamma(T,\rho^\ast(\Omega^0)^\vee);\quad \varphi(\rho^\ast t_{\dP,\tau})\equiv1,\; \varphi(\rho^\ast s_{\dP,\tau})\equiv0\mod \cI_n\}.
\]

The morphism $g_n$ is endowed with an action of $(\cO/p^{n}\cO)^\times$, then both $\V_{Q_{{\bf s}^0}}(\Omega^0,{\bf t}^0)/\dX_r$ and $\V_0(\Omega_0,s_0)/\dX_r$ are equipped with actions of $\cO^\times(1+\cI_n{\rm Res}_{\cO_F/\Z}\G_a)\subseteq{\rm Res}_{\cO/\Z_p}\G_m$ (see \cite[\S 6.3]{BM}).

Since $r\geq n$ by Lemma \ref{weightextends} (with $\lambda={\rm Hdg}^{\frac{p^n}{p-1}}$) the character ${\bf k}_n^0$ extends to a locally analytic character 
\[
{\bf k}_n^0:\cO^\times(1+\cO_F\otimes_{\Z}\cI_n\cO_{\dX_r}\otimes_{\Z_p}\Lambda_n^0)\longrightarrow \cO_{\dX_r}\otimes_{\Z_p}\Lambda_n^0.
\]

\begin{defi} We consider the following sheaves over $\dX_{r}\times\dW_n$:
\begin{eqnarray*}
\cF_n&:=& \left((g_n\circ f^{0})_{\ast}\cO_{\V_{Q_{{\bf s}^0}}(\Omega^0,{\bf t}^0)}\hat\otimes\Lambda_n\right)[{\bf k}_{n}^{\tau_0,0}],\qquad \Omega_0^{{\bf k}_{\tau_0,n}^0}:=\left((g_n\circ f_{0})_{\ast}\cO_{\V_0(\Omega_0,s_0)}\hat\otimes\Lambda_n\right)[{\bf k}_{\tau_0,n}^0], \\
 \Omega^{{\bf k}_{f}}&:=&\left( g_{1, \ast}(\cO_{\mathfrak{IG}_{1}})\hat\otimes\Lambda_n\right)[{\bf k}_{f}].
\end{eqnarray*}
The formal \emph{overconvergent modular sheaf} over $\dX_{r}\times\dW_n$ is defined as 
\[
\Omega^{{\bf k}_{n}}:= \Omega_0^{{\bf k}_{\tau_0,n}^0}\otimes_{\cO_{\dX_{r}\times\dW_n}}\cF_n^\vee\otimes_{\cO_{\dX_{r}\times\dW_n}}\Omega^{{\bf k}_{f}},\qquad\cF_n^\vee:=\Hom_{\dX_{r}\times\dW_n}\left(\cF_n,\cO_{\dX_{r}\times\dW_n}\right).
\]
\end{defi}

\begin{defi} 
A section in $\mathrm{H}^0(\dX_{r}\times\dW_{n}, \Omega^{{\bf k}_{n}})$ is called a \emph{family of locally analytic Overconvergent modular forms}.
\end{defi}

\subsection{Overconvergent modular forms \`a la Katz} \label{ss:overconvergent modular forms a la katz}
Here we give a moduli description of the families of overconvergent modular forms introduced above.
\subsubsection{Notations}\label{NotCD}
 Let $R$ be a complete local $\hat\cO$-algebra. 
 \begin{defi}
Let $k:\cO^{\tau_0\times}\rightarrow R$ be a character and $n \in \N$. We denote by  $C^{k}_n(\cO^{\tau_0},R)$ the $R$-module of the functions $f:\cO^{\tau_0\times}\times \cO^{\tau_0}\rightarrow R$ such that: 
\begin{itemize}
\item $f(tx, ty)=k(t)\cdot f(x, y)$  for each  $t\in\cO^{\tau_0\times}$ and $(x, y) \in \cO^{\tau_0\times}\times \cO^{\tau_0}$;
\item  the function $y\mapsto f(1, y)$ is analytic on the disks $y_0+ p^n\cO^{\tau_0}$  where $y_0$ below to a system of representatives of $\cO^{\tau_0}/p^n\cO^{\tau_0}$.
\end{itemize}
The space of \emph{distributions} is defined by $D^{k}_n(\cO^{\tau_0},R):={\rm Hom}_R(C^{k}_n(\cO^{\tau_0},R),R).$
\end{defi}

\begin{rmk}\label{rmkanalicity}
Note that $C^{k}_n(\cO^{\tau_0},R)\subseteq C^{k}_{n+1}(\cO^{\tau_0},R)$ and if $k=\underline{k}\in\N[\Sigma_F\smallsetminus \{\tau_0\}]$ is a classical weight then $C^{\underline{k}}_0(\cO^{\tau_0},R)$ is the module of analytic functions and naturally contains $\Sym^{\underline{k}}(R^2)$. We obtain a natural projection $D^{\underline{k}}_0(\cO^{\tau_0},R)\rightarrow\Sym^{\underline{k}}(R^2)^\vee$. 
\end{rmk}

We have a natural action of the subgroup $K_0(p)^{\tau_0}\subset\GL_2(\cO^{\tau_0})$ of upper triangular matrices modulo $p$ on $C^{k}_n(\cO^{\tau_0},R)$ and $D^{k}_n(\cO^{\tau_0},R)$ given by: 
\[
(g\ast f)(x,y)= f((x,y)g) \qquad \qquad (g\ast\mu)(f):=\mu(g^{-1}\ast f),
\]
where $g\in K_0(p)^{\tau_0}$, $f\in C^{k}_n(\cO^{\tau_0},R)$ and $\mu\in D^{k}_n(\cO^{\tau_0},R)$. Since $y \mapsto f(1, y)$ is analytic on the disks $y_0+ p^n\cO^{\tau_0}$ this action extends to an action of $K_0(p)^{\tau_0}(1+p^n\M_2(R\otimes_{\Z_p}\cO^{\tau_0}))$.

\subsubsection{Locally analytic distributions}\label{locallanaldist} We are going to describe the elements of 
$H^0(\dX_{r}\times\dW_{n}, \Omega^{{\bf k}_n})$ as rules, extending the classical interpretation due to Katz. As always, we are assuming that $r \geq n$.

Now let $R$ be a $\Lambda_n$-algebra.
Recall that a tuple $(A,\iota,\theta,\alpha^{\dP^0})$ defined over $R$ corresponds to a point in $\dX_r(R)$. We will denote by $w=(f_0,\{(f_\tau,e_\tau)\}_\tau)$ the basis of $\Omega$ such that $f_0$ is a basis of $\Omega_0$ and $\{e_\tau, f_\tau\}$  is a basis of  $\omega_{\dP,\tau}$, where $\dP\neq\dP_0$ and $\tau \in \Sigma_{\mathfrak{p}}$.
As seen in \cite[\S 6.5.2]{BM}, any section $\mu \in H^0(\dX_{r}\times\dW_n, \Omega^{{\bf k}_{n}})$ is characterized by a rule that assigns to each tuple $(A,\iota,\theta,\alpha^{\dP^0},w)$ over $R$, a distribution $\mu(A,\iota,\theta,\alpha^{\dP^0},w)\in D^{{\bf k}_n^{\tau_0}}_n(\cO^{\tau_0},R)$. 
The rule $(A,\iota,\theta,\alpha^{\dP_0},w)\mapsto \mu(A,\iota,\theta,\alpha^{\dP_0},w)$ satisfies:
\begin{itemize}
\item[(B1)] $\mu(A,\iota,\theta,\alpha^{\dP_0},w)$ depends only on the $R$-isomorphism class of $(A,\iota,\theta,\alpha^{\dP_0})$.

\item[(B2)] The formation of $\mu(A,\iota,\theta,\alpha^{\dP_0},w)$ commutes with arbitrary extensions of scalars $R\rightarrow R'$ of $\Lambda_n$-algebras.


\item [(B3-a)] $\mu(A,\iota,\theta,\alpha^{\dP_0}, a^{-1}w)=k_n^{\tau_0}(t)\cdot\mu(A,\iota,\theta,\alpha^{\dP_0},w)$, for all $t\in\Z_p^\times$.

\item [(B3-b)] $g\ast\mu(A,\iota,\theta,\alpha^{\dP_0},wg)=\mu(A,\iota,\theta,\alpha^{\dP_0},w)$, for all $g\in K_0(p)^{\tau_0}$.
\end{itemize}



\begin{rmk}
Note that this description fits with the classical setting of classical Katz modular forms explained in \S \ref{KatzModForm}.
\end{rmk}

\section{Nearly overconvergent modular forms}

Let us consider the sheaf $\cR^1\pi_\ast\Omega_{{\bf A}/\dX}^\bullet$, where ${\bf A}\rightarrow \dX$ is the universal abelian variety. By the moduli interpretation of $\dX$
\[
\cH:=\left(\cR^1\pi_\ast\Omega_{{\bf A}/\dX}^\bullet\right)^{+,2}=\cH_0\oplus\bigoplus_{\dP\neq\dP_0}\omega_\dP.
\]
Notice that $\cH_0$ is endowed with:
\begin{itemize}
\item A connection $\triangledown:\cH_0\rightarrow\cH_0\otimes_{\cO_\dX}\Omega_{\dX/\hat\cO}^1$.

\item A Hodge filtration
\begin{equation}\label{HdgFil}
0\longrightarrow\omega_0 \longrightarrow\cH_0\stackrel{\epsilon}{\longrightarrow}\omega_0^{-1}\longrightarrow 0,
\end{equation}
(see \cite[\S 9.1]{BM}).
\end{itemize} 

Fixing integers $r\geq n$ as above and having the morphisms $\mathfrak{IG}_{r,n} \rightarrow\dX_{r,n}\rightarrow\dX_{r}\rightarrow\dX$, we can base-change the triple $(\cH_0,\triangledown,{\rm Fil}^\bullet)$ over $\mathfrak{IG}_{r,n}$ and denote it the same way. 

Write $\cH_0^\sharp:=\Omega_0+\underline{\delta}^p\cH_0$. Since $\underline{\delta}$ is a locally free $\cO_{\mathfrak{IG}_{r,n}}$-module of rank 1, $\cH_0^\sharp$ is a $\cO_{\mathfrak{IG}_{r,n}}$-module of rank 2 with Hodge filtration 
\begin{equation}\label{exseqHw}
0\longrightarrow\Omega_0 \longrightarrow\cH_0^\sharp\stackrel{\epsilon}{\longrightarrow}\underline{\delta}^p\omega_0^{-1}\longrightarrow 0.
\end{equation}
This implies that $s_0$ is also a marked section of $\cH_0^\sharp$. 
Similarly as in \cite[Lemma 6.1]{AI17} one can prove that the construction of the exact sequence \eqref{exseqHw} is functorial.

\begin{lemma}
For any lift $\tilde s_0\in \Omega_0$ of $s_0$ and any generator $D$ of the space of derivations of $\dX$, the subspace $Q={\rm Hdg}\langle\triangledown(D)(\tilde s_0)\rangle\;({\rm mod} \,\cI_n)$ defines a canonical direct summand $\overline{\cH_0^\sharp}=\langle s_0\rangle\oplus Q$.
\end{lemma}
\begin{proof}
First, let me remark that the definition of $Q$ does not depend on the choice of $D$ nor the lift $\tilde s_0$. Indeed, for any $\alpha\in\cI_n=p^n\underline{\delta}^{-p^n}$
\[
{\rm Hdg}\triangledown(D)(\tilde s_0+\alpha s)={\rm Hdg}(\triangledown(D)(\tilde s_0)+D\alpha s+\alpha\triangledown(D)(s))\equiv {\rm Hdg}\triangledown(D)(\tilde s_0)+{\rm Hdg}D\alpha s\mod\,\cI_n,
\]
But $D\alpha\in p^{2n}\underline{\delta}^{-p^n-1}=p^n\underline{\delta}^{-1}\cI_n$, thus ${\rm Hdg}D\alpha s=\underline{\delta}^{p-1}D\alpha s\equiv 0$ modulo $\cI_n$.

In order to prove the claim, we will work locally. Let $\rho:\Spf(R)\rightarrow \mathfrak{IG}_{r,n}$ be a morphism of formal schemes without $p$-torsion such that $\rho^\ast\omega_0$, and $\rho^\ast\cH_0$ are free $R$-modules of rank 1 and 2, respectively. We choose basis $\rho^\ast\omega_0=R f$, $\rho^\ast\cH_0=Rf+Re$. Recall that the Kodaira-Spencer isomorphism $KS$ is obtained from restricting $\triangledown$ to $\omega_0$ and composing with the morphism $\epsilon$ of \eqref{HdgFil}. Thus, if the derivation $D$ is dual to $\Theta=KS(f,e)\in \Omega_{\dX/\Z_p}^1$, we have that 
\[
\triangledown(D)(f)=af+e,\quad\mbox{for some }a\in R.
\] 
Assume also that $\rho^\ast \underline\delta=\delta R$. Moreover, we can choose $\delta$ so that $\tilde s_0=\delta f$. Hence we obtain
\[
\triangledown(D)(\tilde s_0)=\triangledown(D)(\delta f)=(D\delta +\delta a)f+\delta e.
\]
Since $\rho^\ast{\rm Hdg}=\delta^{p-1}R$, we obtain
\[
{\rm Hdg}\langle\triangledown(D)(\tilde s_0)\rangle= \delta^{p-1}\triangledown(D)(\tilde s_0)R=((\delta^{p-2}D\delta +\delta^{p-1} a)\delta f+\delta^p e)R\subset \cH_0^\sharp.
\]
Since $\rho^\ast\cH_0^\sharp=\delta fR\oplus \delta^pe E$, we obtain that $\overline{\cH_0^\sharp}=\langle s_0\rangle\oplus Q$. 
\end{proof}
The above lemma provides all the ingredients to construct the formal vector bundle $\V_Q(\cH_0^\sharp,s_0)\rightarrow\dX_r$.
Similarly as in \cite[\S 6.3]{BM}, we define an action of $\Z_p^\times(1+\cI_n\G_a)$ on $\V_Q(\cH_0^\sharp,s_0)\rightarrow\dX_r$ as follows: For any $\dX_r$-scheme $T$
\[
\V_Q(\cH_0^\sharp,s_0)(T)=\{(\rho,\varphi)\in \mathfrak{IG}_{r,n}(T)\times\Gamma(T,\rho^\ast(\cH_0^\sharp)^\vee);\quad  \varphi(\rho^\ast s_0)\equiv1,\; \varphi(\rho^\ast Q)\equiv 0\mod \cI_n\},
\]
and the action of $\lambda(1+\gamma)\in\Z_p^\times(1+\cI_n\G_a)$ is given by $\lambda(1+\gamma)\ast(\rho,\varphi):=(\lambda\rho,\lambda(1+\gamma)\ast\varphi)$, where $\lambda\ast\varphi(w)=\varphi(\lambda w)$.

\begin{defi} We consider the following sheaf over $\dX_{r}\times\dW_n$:
\[
\W_0^{{\bf k}_{\tau_0,n}^0}:=\left((g_n\circ f_{0})_{\ast}\cO_{\V_Q(\cH_0^\sharp,s_0)}\hat\otimes\Lambda_n\right)[{\bf k}_{\tau_0,n}^0], 
 \]
 corresponding to sections such that $t\ast s={\bf k}_{\tau_0,n}^0(t)s$, for all $t\in \Z_p^\times(1+\cI_n\G_a)$.
The formal \emph{nearly overconvergent modular sheaf} over $\dX_{r}\times\dW_n$ is defined as 
\[
\W^{{\bf k}_{n}}:= \W_0^{{\bf k}_{\tau_0,n}^0}\otimes_{\cO_{\dX_{r}\times\dW_n}}\cF_n^\vee\otimes_{\cO_{\dX_{r}\times\dW_n}}\Omega^{{\bf k}_{f}}.
\]
\end{defi}

\begin{defi} 
A section in $\mathrm{H}^0(\dX_{r}\times\dW_{n}, \W^{{\bf k}_{n}})$ is called a \emph{family of locally analytic nearly overconvergent modular forms}.
\end{defi}

The inclusion $\Omega_0\subset \cH_0^\sharp$ provides a filtration of locally free sheaves with marked sections $(\Omega_0,s_0)\hookrightarrow (\cH_0^\sharp,s_0)$. By \S \ref{filFVB}, the sheaf $(g_n\circ f_{0})_{\ast}\cO_{\V_Q(\cH_0^\sharp,s_0)}$ has a canonical filtration ${\rm Fil}_\bullet(g_n\circ f_{0})_{\ast}\cO_{\V_Q(\cH_0^\sharp,s_0)}:=(g_n\circ f_{0})_{\ast}{\rm Fil}_\bullet\cO_{\V_Q(\cH_0^\sharp,s_0)}$. 

Analogously as in \cite[Theorem 3.11.]{AI17}, the action of $\Z_p^\times(1+\cI_n\G_a)$ on $(g_n\circ f_{0})_{\ast}\cO_{\V_Q(\cH_0^\sharp,s_0)}$ preserves the filtration $(g_n\circ f_{0})_{\ast}{\rm Fil}_\bullet\cO_{\V_Q(\cH_0^\sharp,s_0)}$. Moreover, if we define 
\[
{\rm Fil}_n\W_0^{{\bf k}_{\tau_0,n}^0}:=\left((g_n\circ f_{0})_{\ast}{\rm Fil}_n\cO_{\V_Q(\cH_0^\sharp,s_0)}\hat\otimes\Lambda_n\right)[{\bf k}_{\tau_0,n}^0]\subset \left((g_n\circ f_{0})_{\ast}\cO_{\V_Q(\cH_0^\sharp,s_0)}\hat\otimes\Lambda_n\right)[{\bf k}_{\tau_0,n}^0]=\W_0^{{\bf k}_{\tau_0,n}^0}.
\]
Similarly as in \cite[\S 3.3.1]{AI17} we have
\begin{itemize}
\item[(i)] ${\rm Fil}_n\W_0^{{\bf k}_{\tau_0,n}^0}$ is a locally free $\cO_{\dX_{r}\times\dW_n}$-module for the Zariski topology on $\dX_{r}\times\dW_n$;

\item[(ii)] $\W_0^{{\bf k}_{\tau_0,n}^0}$ is the $p$-adic completion of $\lim_n{\rm Fil}_n\W_0^{{\bf k}_{\tau_0,n}^0}$;

\item[(iii)] ${\rm Fil}_0\W_0^{{\bf k}_{\tau_0,n}^0}\simeq \Omega_0^{{\bf k}_{\tau_0,n}^0}$. 
\end{itemize}
Moreover, if we define ${\rm Fil}_n\W^{{\bf k}_{n}}:={\rm Fil}_n\W_0^{{\bf k}_{\tau_0,n}^0}\otimes_{\cO_{\dX_{r}\times\dW_n}}\cF_n^\vee\otimes_{\cO_{\dX_{r}\times\dW_n}}\Omega^{{\bf k}_{f}}$, we obtain an increasing filtration $\{{\rm Fil}_n\W^{{\bf k}_{n}}\}_n$ by directs summands such that claims $(ii)$ and $(iii)$ hold replacing $\Omega_0^{{\bf k}_{\tau_0,n}^0}$ with $\Omega^{{\bf k}_{n}}$.

Finally, if $k\in \N$ is a classical weight then we have a canonical identification
\[
{\rm Sym}^k(\cH_0)[1/p]={\rm Fil}_k(\W_0^k)[1/p]:=k^\ast \left({\rm Fil}_k(\W_0^{{\bf k}^0_{\tau_0,n}})\otimes_{\cO_{\dX_{r}\times\dW_n}}\Omega^{{\bf k}_{f}}\right)[1/p].
\]
as sheaves on the corresponding adic space $\cX_{r}\times\cW_n$ compatible with the Hodge filtration on ${\rm Sym}^k(\cH_0)$.

\begin{rmk}
We omit the proof of the above facts because it is completely analogous to \cite[\S3.3.3]{AI17} using results of \cite{AIP-halo} and \S \ref{filFVB}. 
\end{rmk}

\subsection{Gauss-Manin connections}


Consider the morphism of adic spaces $\overline{\mathcal{IG}}_{r,n}\rightarrow\mathcal{IG}_{r,n}$ defined by the trivializations $\cG_0[\dP_0^n]\simeq(\Z/p^n\Z)^2$ compatible with the trivializations $D_n\simeq \Z/p^n\Z$.
Let $\overline{\mathfrak{IG}}_{r,n}\rightarrow\mathfrak{IG}_{r,n}$ be the normalization. It follows from \cite[Proposition 6.3]{AI17} and \S \ref{connFVB} that over $\overline{\mathfrak{IG}}_{r,n}$ the sheaf 
\[
\overline{\W}_0^{{\bf k}_{\tau_0,n}^0}:=\left((g_n\circ f_{0})_{\ast}\cO_{\V_0(\cH_0^\sharp,s_0)}\hat\otimes\Lambda_n\right)[{\bf k}_{\tau_0,n}^0], 
\]
admits an integrable connection relatively to $\Lambda_{\tau_0,n}^0$ for which ${\rm Fil}_\bullet\overline{\W}_0^{{\bf k}_{\tau_0,n}^0}$ satisfies Griffiths' transversality.

Recall that, by \eqref{univK}, the universal character ${\bf k}_{\tau_0,n}^0$ is given by
\[
{\bf k}_{\tau_0,n}^0(\beta)=\sum_{j=0}^{p^{n-1}-1}(1+T)^j\cdot 1_{\exp(jp)+p^{n}\Z_p}(\beta)\cdot\sum_{i\geq 0}\binom{w_n}{i}\left(\frac{\beta-\exp(jp)}{\exp(jp)}\right)^i.
\]
where $w_n=p^{-1}\log(1+T)\in p^{-n+1}\Lambda_{\tau_0,n}^0$.

\begin{thm}\label{GMthm}
The above connection descends to an integrable connection
\[
\bigtriangledown_{{\bf k}_{\tau_0,n}}:\W_0^{{\bf k}_{\tau_0,n}^0}\longrightarrow \W_0^{{\bf k}_{\tau_0,n}^0}\otimes_{\cO_{\dX}}\underline{\delta}^{-c_n}\Omega^1_{\dX/\Lambda_n};\qquad \dX:=\dX_{r}\times\dW_n,
\]
for some $c_n\in\N$ depending on $n$, such that the induced $\cO_{\dX}$-linear map on the $n$ graded piece
\[
{\rm Gr}_h(\bigtriangledown_{{\bf k}_{\tau_0,n}}):{\rm Gr}_n(\W_0^{{\bf k}_{\tau_0,n}^0})[1/p]\longrightarrow {\rm Gr}_{h+1}(\W_0^{{\bf k}_{\tau_0,n}^0})\otimes_{\cO_{\dX}}\Omega^1_{\dX/\Lambda_n}[1/p]
\]
is an isomorphism times $w_n-h$. 
\end{thm}
\begin{proof}
Let us work locally. Let $\rho:\Spf(R)\rightarrow \mathfrak{IG}_{r,n}$ be a morphism of formal schemes without $p$-torsion such that $\rho^\ast\omega_0$, and $\rho^\ast\cH_0$ are free $R$-modules of rank 1 and 2, respectively. We choose basis $\rho^\ast\omega_0=R f$, $\rho^\ast\cH_0=Rf+Re$ and assume that $\underline{\delta}=\delta R$. Moreover, we assume that $\delta f\equiv s_0$ and $\langle\delta^p e\rangle\equiv Q$ modulo $\cI_n$. Thus, if the derivation $d$ is dual to the image of the Kodaira-Spencer isomorphism $\Theta=KS(f,e)\in \Omega_{R/\Z_p}^1$, we have that 
\begin{equation}\label{eqaux1}
\triangledown(d)(\delta f)=af+\delta e,\quad\mbox{for some }a\in \cI_n.
\end{equation}

Computing things in $R[1/\delta]$, the setting becomes ordinary. It can be deduced from \cite[Main Theorem 4.3.2]{K81} that any section $s$ in the image of the map ${\rm dlog}:T_p\cG_0\rightarrow \Omega_0$ satisfies $\triangledown(d)^2(s)=0$. Since $\delta f$ is in the image of ${\rm dlog}$ modulo $\cI_n$, we have $\delta f=s+s_0$, where $s\in {\rm dlog}(T_p\cG_0)$ and $s_0\in \delta^{-a_n}\cI_n\Omega_0$, for some $a_n\in\N$. Since $\triangledown(d)\delta^{-a_n}\cI_n\Omega_0\subseteq \delta^{-a_n-1}\cI_n\Omega_0$, 
we deduce that $\triangledown(d)^2(\delta f)\in \delta^{-b_n}\cI_n$, for some $b_n\in\N$. Applying the Leibniz rule, we compute that
\[
\triangledown(d)(\delta^p e)=(\delta^{p-3}a(d\delta-a)-\delta^{p-2}da)\delta f+\delta^{-1}((p-1)d\delta-a)\delta^p e+\delta^{p-1}\triangledown(d)^2(\delta f).
\]
Since $da\in p^n\delta^{-1}\cI_n$, this implies that, for some $c_n\in\N$,
\[
\triangledown(\delta^{c_n}d)(\delta f)=A\delta f+B\delta^p e,\qquad \triangledown(\delta^{c_n}D)(\delta^p e)=C\delta f+D\delta^p e,
\]
where $A,C\in\cI_n$ and $B,D\in R$.

Fix a generator $\alpha\in\cI_n$.
As explained in \cite[\S 6.4]{BM}, the sections $\rho^\ast\cO_{\V_Q(\cH_0^\sharp,s_0)}(\Spf(R))$ are given by
\[
\rho^\ast\cO_{\V(\cH_0^\sharp)}=R\langle X,Y\rangle\longrightarrow R\langle Z,T\rangle=\rho^\ast\cO_{\V_Q(\cH_0^\sharp,s_0)},\qquad X\mapsto1+\alpha Z,\quad Y\mapsto\alpha T,
\]
where $X$ corresponds to $\delta f$ and $Y$ corresponds to $\delta^p e$. This implies that the connection on $\rho^\ast\cO_{\V(\cH_0^\sharp)}$ is given by $\triangledown(\delta^{c_n}d)(X)=AX+BY$, $\triangledown(\delta^{c_n}d)(Y)=CX+DY$. We observe that this connection descends to $\rho^\ast\cO_{\V_Q(\cH_0^\sharp,s_0)}$, indeed we can define
\[
\triangledown(\delta^{c_n}d)(Z)=\alpha^{-1}A(1+\alpha Z)+BT,\qquad \triangledown(\delta^{c_n}d)(T)=\alpha^{-1}C(1+\alpha Z)+DT,
\]
since $\alpha^{-1}A,\alpha^{-1}C\in R$. Thus, by Lemma \ref{lemOmegaOmega}, after possibly enlarging $c_n$ we obtain
\[
\bigtriangledown:(g_n\circ f_0)_\ast\cO_{\V_Q(\cH_0^\sharp,s_0)}\longrightarrow (g_n\circ f_0)_\ast\cO_{\V_Q(\cH_0^\sharp,s_0)}\otimes_{\cO_{\dX}}\underline{\delta}^{-c_n}\Omega^1_{\dX/\Lambda_n}. 
\]
In \cite[\S 6.4]{BM} we have also the local description of $\rho^\ast\W_0^{{\bf k}_{\tau_0,n}^0}$: It is given by 
\[
\rho^\ast\W_0^{{\bf k}_{\tau_0,n}^0}=P_n(Z)\cdot R\left\langle\frac{T}{1+\alpha Z}\right\rangle,\qquad P_n(Z)=\sum_{i\geq 0}\binom{w_n}{i}\alpha^i Z^i,
\]
being $P_n$ the analytic extension of ${\bf k}_{\tau_0,n}^0$ around $1+\cI_nR\langle Z\rangle$. We compute using Leibniz rule:
\begin{eqnarray*}
&&\triangledown(\delta^{c_n}d)\left(P_n(Z)\left(\frac{T}{1+\alpha Z}\right)^m\right)=\\
&&=P_n(Z)\left(\frac{T}{1+\alpha Z}\right)^{m-1}\left((w_n-m)\alpha B\left(\frac{T}{1+\alpha Z}\right)^2+((w_n-m)A+mD)\left(\frac{T}{1+\alpha Z}\right)+m\alpha^{-1}C\right).
\end{eqnarray*}
From this we deduce the first assertion.

Moreover, the above computation shows that the induced morphism in graded peaces is given by multiplication by $(w_n-m)\alpha B$, and recall by \eqref{eqaux1} that $B=\delta^{c_n+1-p}$. Hence, once we invert $p$ (and therefore $\underline{\delta}$) we obtain the second assertion.
\end{proof}

Together with the connection on $\Omega^{{\bf k}_{f}}[1/p]$ provided by the universal derivation on $\mathfrak{IG}_{r,n}$, the above $\bigtriangledown_{{\bf k}_{\tau_0,n}}$ induces a connection on $\W^{{\bf k}_{\tau_0,n}}:=\W_0^{{\bf k}^0_{\tau_0,n}}\otimes_{\cO_{\dX_{r}\times\dW_{\tau_0,n}}}\Omega^{{\bf k}_{f}}[1/p]$ compatible with the identification 
\[
{\rm Sym}^k(\cH_0)[1/p]=k^\ast \left({\rm Fil}_k(\W_0^{{\bf k}^0_{\tau_0,n}})\otimes_{\cO_{\dX_{r}\times\dW_{\tau_0,n}}}\Omega^{{\bf k}_{f}}\right)[1/p],
\]
and the Gauss-Manin connection on ${\rm Sym}^k(\cH_0)[1/p]$.


\subsection{Hecke operator $U_{\dP_0}$}

Let us consider the morphisms $p_1,p_2:\dX_{r+1}\rightarrow \dX_{r}$ defined on the universal abelian variety ${\bf A}\mapsto {\bf A}$ and ${\bf A}\mapsto {\bf A}'={\bf A}/\tilde C_1$, where $\tilde C_1\subset {\bf A}[\dP_0]$ is the $\cO_D$-submodule generated by the canonical subgroup $C_1$.
By \cite[Proposition 3.24]{AI17} and \cite[Proposition 6.5]{AI17}, the dual isogeny $\lambda:{\bf A}'\rightarrow {\bf A}$ induces morphisms of $\cO_{\dX}$-modules
\[
\cU:p_{2,\ast}p_1^\ast(\W_0^{{\bf k}^0_{\tau_0,n}})\longrightarrow p_{2,\ast}p_2^\ast(\W_0^{{\bf k}^0_{\tau_0,n}}),\qquad \cU:p_{2,\ast}p_1^\ast(\W^{{\bf k}_{\tau_0,n}})\longrightarrow p_{2,\ast}p_2^\ast(\W^{{\bf k}_{\tau_0,n}}),
\]
which commute with $\bigtriangledown_{{\bf k}_{\tau_0,n}}$ and preserve filtrations.

Similarly as in \cite[Proposition 3.3]{AIP-halo}, the morphism $p_2:\dX_{r+1}\rightarrow \dX_{r}$ is finite flat of degree $p$ hence there is a well defined trace map with respect to $\cO_{\dX_r}\rightarrow p_{2,\ast}(\cO_{\dX_{r+1}})$. We define the $U_{\dP_0}$-operator
\[
U_{\dP_0}:H^0(\dX,\W^{{\bf k}_{\tau_0,n}})\stackrel{\cU\circ p_1^\ast}{\longrightarrow}H^0(\dX,p_{2,\ast}p_2^\ast(\W^{{\bf k}_{\tau_0,n}}))\stackrel{\frac{1}{p}{\rm Tr}_{p_2}}{\longrightarrow}H^0(\dX,\W^{{\bf k}_{\tau_0,n}})[1/p].
\]


\begin{lemma}\label{lemFilm}
The morphism induced by $U_{\dP_0}$ on $H^0(\dX,\W^{{\bf k}_{\tau_0,n}}/{\rm Fil}_m\W^{{\bf k}_{\tau_0,n}})$ is nilpotent modulo $p^{[m/2]-1}$, for $m\geq p$ where $[\cdot]$ denotes the integral part.
\end{lemma} 
\begin{proof}
Analogously to \cite[Proposition 3.24]{AI17} we will prove that the morphism $\cU$ induces a morphism on the $m$th graded part of the filtration that is zero modulo $p^{[m/p]}$, this automatically implies the claim.

By \cite[Lemma 6.4]{AI17} the map induced on de Rham cohomology $\lambda:\cH_0^{\sharp}\rightarrow\cH_{\cG_0'}^{\sharp}$ (where $\cG_0'$ is the $p$-divisible group associated with ${\bf A}'$) gives an isomorphism $\Omega_0\simeq\Omega_0':=\Omega_{\cG_0'}^1$ and identifies $\cH_0^{\sharp}/\Omega_0=\underline{\delta}^p\omega_0^\vee$ with $p\underline{\delta}^{1-p^2}\cH_{\cG_0'}^{\sharp}/\Omega_0'$. Using the description of ${\rm Gr}_m\W_0^{{\bf k}_{\tau_0,n}^0}$ provided in \eqref{gradpiec} we deduce that the induced morphism on graded parts is zero modulo $p^m\underline{\delta}^{(1-p^2)m}$. By construction $p{\rm Hdg}^{-p^{r+2}}$ is a well defined section of $\dX_{r+1}$, since $r\geq 1$, 
\[
\left(p^m\underline{\delta}^{(1-p^2)m}\right)^2=p^{2m}{\rm Hdg}^{-2(p+1)m}\subset p^m(p{\rm Hdg}^{-p^3})^m\subset p^m\cO_{\dX_{r+1}},
\]
hence the result follows.
\end{proof}

\begin{rmk}\label{rmkonheckeops}
We note that the above definition of $U_{\dP_0}$ fits with the one given in \cite[\S 7.3]{BM} acting on $H^0(\dX,\Omega^{{\bf k}_n})$. Moreover, in \cite[\S 7]{BM} one can also find the definition of the Hecke operators $V_{\dP_0}$ and $U_{\dP}$, for $\dP\neq \dP_0$, acting on $H^0(\dX,\Omega^{{\bf k}_n})$.
\end{rmk}

\subsection{De Rham cohomology with coefficients in $\W_0^{{\bf k}^0_{\tau_0,n}}$ and the overconvergent projection}

The multiplicative law on $(g_n\circ f_0)_\ast\cO_{\V_0(\cH_0^\sharp,s_0)}$ provides a morphism 
\[
{\rm Fil}_2\W^2\times \W^{{\bf k}_{\tau_0,n}}\longrightarrow \W^{{\bf k}_{\tau_0,n}+2}.
\]
Moreover, by the Kodaira-Spenser isomorphism we can identify $\Omega_{\cX_r}^1$ with ${\rm Fil}_2\W^2$ in the adic space $\cX_{r}$. Hence, the connection $\bigtriangledown_{{\bf k}_{\tau_0,n}}$ can be reinterpreted as the morphism of de Rham complexes of sheaves on the adic space $\cX_r\times\cW_n$
\[
\W_{{\bf k}_{\tau_0,n}}^{\bullet}:\W^{{\bf k}_{\tau_0,n}}\stackrel{\bigtriangledown_{{\bf k}_{\tau_0,n}}}{\longrightarrow}\W^{{\bf k}_{\tau_0,n}+2};\qquad {\rm Fil}_m(\W_{{\bf k}_{\tau_0,n}}^{\bullet}):{\rm Fil}_m(\W^{{\bf k}_{\tau_0,n}})\stackrel{\bigtriangledown_{{\bf k}_{\tau_0,n}}}{\longrightarrow}{\rm Fil}_{m+1}(\W^{{\bf k}_{\tau_0,n}+2}).
\]   
Let us denote by $H_{dR}^i(\cX_r\times\cW_n,\bullet)$ the $i$-th hypercohomology group of such de Rham complexes.


The same identification as above allows us to consider the Gauss-Manin operator given by Griffith's transversality 
\[
\bigtriangledown_m:\omega_0^{k-m}\otimes\Sym^m\cH_0\longrightarrow \omega_0^{k-m+1}\otimes\Sym^{m+1}\cH_0.
\]
Notice that the sheaves $\omega_0^{k-m}\otimes\Sym^m\cH_0$ define a filtration ${\rm Fil}_m$ of $\Sym^k\cH_0$. Moreover, on graded pieces the connection $\bigtriangledown_m$ defines an isomorphism times $k-m$. This implies that 
\[
\bigtriangledown_k(\Sym^k\cH_0)\subseteq {\rm Fil}_{k}\Sym^{k+2}\cH_0={\rm Fil}_{k-1}\Sym^{k+2}\cH_0+{\rm Im}\bigtriangledown_{k-1}=\cdots={\rm Fil}_{0}\Sym^{k+2}\cH_0+{\rm Im}\bigtriangledown_{k-1}.
\]
Since ${\rm Gr}_m(\Sym^k\cH_0)\simeq \omega_0^{k-2m}$, we obtain that $\bigtriangledown_k$ defines a morphism
\[
\theta^{k+1}:H^0(\cX_r,\omega_0^{-k})\longrightarrow H^0(\cX_r,\omega_0^{k+2}).
\]
As shown in \cite[\S 8.2]{BM}, the classical computations on $q$-expansions also apply on this setting considering Serre-Tate coordinates.
This implies that the exactly same proof of \cite[Lemma 3.32]{AI17} works showing that $H_{dR}^1(\cX_r\times\cW_n,{\rm Fil}_m(\W_{{\bf k}_{\tau_0,n}}^{\bullet}))$ lies in the exact sequence
\begin{equation}\label{exseqOW}
0\longrightarrow H^0(\cX_r\times\cW_n,\Omega_0^{{\bf k}_{\tau_0,n}+2})\longrightarrow H_{dR}^1(\cX_r\times\cW_n,{\rm Fil}_m(\W_{{\bf k}_{\tau_0,n}}^{\bullet}))\longrightarrow\bigoplus_{i=0,\rho}^mH^0(\cX_r\times\cW_n,j_{\rho,\ast}(\omega_0)^{-i})\longrightarrow 0
\end{equation}
where $j_\rho$ is the closed immersion $\cX_r\hookrightarrow \cX_r\times\cW_n$ defined by a point such that $w_n=i$. Moreover, the torsion free part $H_{dR}^1(\cX_r\times\cW_n,{\rm Fil}_m(\W_{{\bf k}_{\tau_0,n}}^{\bullet}))^{tf}$ lies in the exact sequence
\[
0\longrightarrow H^0(\cX_r\times\cW_n,\Omega_0^{{\bf k}_{\tau_0,n}+2})\longrightarrow H_{dR}^1(\cX_r\times\cW_n,{\rm Fil}_m(\W_{{\bf k}_{\tau_0,n}}^{\bullet}))^{tf}\longrightarrow\bigoplus_{i=0,\rho}^m\theta^{i+1}\left(H^0(\cX_r\times\cW_n,j_{\rho,\ast}(\omega_0)^{-i})\right)\longrightarrow 0
\]
where the morphisms are equivariant for the action of $U_{\dP_0}$ and the action of $U_{\dP_0}$ on $H^0(\cX_r\times\cW_n,j_{\rho,\ast}(\omega_0)^{-i})$ is divided by $p^{i+1}$.

Let $I_m$ be the ideal $\prod_{i=0}^m(w_n-i)$ and $\Lambda_{I_{m}}$ the localization of $\Lambda_{\tau_0,n}$ at $I_m$. The above exact sequence \eqref{exseqOW} provides an isomorphism
\begin{equation}\label{Hn}
H_m^{\dagger}:H_{dR}^1(\cX_r\times\cW_n,{\rm Fil}_m(\W_{{\bf k}_{\tau_0,n}}^{\bullet}))\otimes_{\Lambda_{\tau_0,n}}\Lambda_{I_{m}}\longrightarrow H^0(\cX_r\times\cW_n,\Omega_0^{{\bf k}_{\tau_0,n}+2})\otimes_{\Lambda_{\tau_0,n}}\Lambda_{I_{m}}.
\end{equation}

Since ${\rm Fil}_m\W^{{\bf k}_{n}}$ is coherent and $U_{\dP_0}$ is compact, the usual discussion on slope decompositions applies to $H^0(\cX_r\times\cW_n,{\rm Fil}_m\W^{{\bf k}_{n}})$. Hence, given a finite slope $h\geq 0$ we have, locally on the weight space, a slope $h$ decomposition. This implies that $H_{dR}^i(\cX_r\times\cW_n,{\rm Fil}_m(\W_{{\bf k}_{\tau_0,n}}^{\bullet}))$ admits a slope $h$ decomposition.
Due to Lemma \ref{lemFilm}, the operator $U_{\dP_0}$ on $H_{dR}^i(\cX_r\times\cW_n,\left(\W/{\rm Fil}_m(\W)\right)^{\bullet})$ is divisible by $p^{h+1}$ for large enough $m$, where $\left(\W/{\rm Fil}_m(\W)\right)^{\bullet}$ is the quotient complex of $\W_{{\bf k}_{\tau_0,n}}^\bullet$ and ${\rm Fil}_m(\W_{{\bf k}_{\tau_0,n}}^\bullet)$. By the long exact sequence of de Rham complexes this implies that 
\begin{equation}\label{HdRWFil}
H_{dR}^i(\cX_r\times\cW_n,{\rm Fil}_m(\W_{{\bf k}_{\tau_0,n}})^{\bullet})^{\leq h}\simeq H_{dR}^i(\cX_r\times\cW_n,\W_{{\bf k}_{\tau_0,n}}^{\bullet})^{\leq h}.
\end{equation}
The above isomorphism together with \eqref{exseqOW} provides a morphism 
\begin{equation}
H^{\dagger}:H_{dR}^1(\cX_r\times\cW_n,\W_{{\bf k}_{\tau_0,n}}^{\bullet})^{\leq h}\otimes_{\Lambda_{\tau_0,n}}\Lambda_{I_{m}}\longrightarrow H^0(\cX_r\times\cW_n,\Omega_0^{{\bf k}_{\tau_0,n}+2})^{\leq h}\otimes_{\Lambda_{\tau_0,n}}\Lambda_{I_{m}},
\end{equation}
where $m$ is an natural number depending on $h$.

The spectral theory developed in \cite[Appendice B]{AIP-halo} provides a overconvergent projection $e_{\leq h}$ onto slope $\leq h$ subspace, this gives rise to the composition:
\begin{equation}\label{Hdh}
\xymatrix{
H^{0}(\cX_r\times\cW_n,\W^{{\bf k}_{\tau_0,n}})\otimes_{\Lambda_{\tau_0,n}}\Lambda_{I_{m}}\ar[d]\ar[r]^{H^{\dagger,\leq h}} &H^0(\cX_r\times\cW_n,\Omega_0^{{\bf k}_{\tau_0,n}})^{\leq h}\otimes_{\Lambda_{\tau_0,n}}\Lambda_{I_{m}}\\
H_{dR}^1(\cX_r\times\cW_n,\W_{{\bf k}_{\tau_0,n}-2}^{\bullet})\otimes_{\Lambda_{\tau_0,n}}\Lambda_{I_{m}}\ar[r]^{e_{\leq h}} & H_{dR}^1(\cX_r\times\cW_n,\W_{{\bf k}_{\tau_0,n}-2}^{\bullet})^{\leq h}\otimes_{\Lambda_{\tau_0,n}}\Lambda_{I_{m}}\ar[u]^{H^\dagger}
}
\end{equation}
where the left down arrow is the natural projection.

\section{Triple products and $p$-adic L-functions}

\subsection{Triple products at $\tau_0$}

In \cite[Theorem 4.11]{BM}, it is shown that there exists a morphism
\begin{eqnarray*}
t_{k_1,k_2,k_3}:H^0(\cX_r,\omega_0^{k_1})\times H^0(\cX_r,\omega_0^{k_2})&\longrightarrow& H^0(\cX_r,\omega_0^{k_3});\\
t_{k_1,k_2,k_3}(s_1,s_2)&:=&\sum_{j=0}^m (-1)^j\binom{m}{j}\binom{M-2}{k_1+j-1}\bigtriangledown_{k_1}^j(s_1)\bigtriangledown^{m-j}_{k_2}(s_2),
\end{eqnarray*}
where $k_1+k_2+k_3$ is even, $k_3\geq k_1+k_2$, $M=\frac{k_1+k_2+k_3}{2}$, and $m=\frac{k_3-k_1-k_2}{2}$. The aim of this section is to define a triple product
\[
{\bf t}:H^0(\cX_r\times\cW_{n_1},\Omega_0^{{\bf k}_{\tau_0,n_1}})\times H^0(\cX_r\times\cW_{n_2},\Omega_0^{{\bf k}_{\tau_0,n_2}})\longrightarrow H^0(\cX_b\times\cW_{n_1}\times\cW_{n_2}\times\cW_{n_3},\Omega_0^{{\bf k}^\ast_{\tau_0,n_3}})^{\leq h},
\]
where ${\bf k}^\ast_{\tau_0,n_3}:\Z_p^\times\rightarrow \Lambda_{\tau_0,n_1}\hat\otimes\Lambda_{\tau_0,n_2}\hat\otimes\hat\Lambda_{\tau_0,n_3}$ is given by ${\bf k}^\ast_{\tau_0,n_3}={\bf k}_{\tau_0,n_1}\otimes{\bf k}_{\tau_0,n_2}\otimes{\bf k}^2_{\tau_0,n_3}$, and $b>r$ is an integer depending on $r\geq {\rm max}\{n_1,n_2,n_3\}$. Such triple product ${\bf t}$ when evaluated at $(k_1,k_2,m)\in(\N^3)\cap (\cW_{n_1}\times\cW_{n_2}\times\cW_{n_3})(\Q_p)$ must coincide with $t_{k_1,k_2,k_3}$ up to constant, where $k_3=k_1+k_2+2m$. 

Given $s_1\in H^0(\cX_r\times\cW_{n_1},\Omega_0^{{\bf k}_{\tau_0,n_1}})$, $s_2\in H^0(\cX_r\times\cW_{n_2},\Omega_0^{{\bf k}_{\tau_0,n_2}})$, and $m\in\N$ we can define using the multiplicative structure of $(g_n\circ f_0)_\ast \cO_{\V_0(\cH_0^\sharp,s_0)}$
\[
\bigtriangledown_{{\bf k}_{\tau_0,n_1}}^{m}(s_1)\cdot s_2\in H^0(\cX_r\times\cW_{n_1,n_2},{\rm Fil}_{m}\W^{{\bf k}_{\tau_0,n_1}+{\bf k}_{\tau_0,n_2}+2m}),\qquad \cW_{n_1,n_2}:=\cW_{n_1}\times\cW_{n_2}.
\]
Write ${\bf k}_{n_1,n_2,m}:={\bf k}_{\tau_0,n_1}+{\bf k}_{\tau_0,n_2}+2m$, $\Lambda_{n_1,n_2}=\Lambda_{\tau_0,n_1}\hat\otimes\Lambda_{\tau_0,n_2}$ and
\[
\Theta_m(s_1,s_2)\in H^0(\cX_r\times\cW_{n_1,n_2},\Omega_0^{{\bf k}_{n_1,n_2,m}})\otimes_{\Lambda_{n_1,n_2}}\Lambda_{I_{m}}
\]
for the image of $\bigtriangledown_{{\bf k}_{\tau_0,n_1}}^{m}(s_1)\cdot s_2$ under the composition of the natural projection
\[
H^0(\cX_r\times\cW_{n_1,n_2},{\rm Fil}_{m}\W^{{\bf k}_{n_1,n_2,m}})\longrightarrow H_{dR}^1(\cX_r\times\cW_{n_1,n_2},{\rm Fil}_{m-1}(\W_{{\bf k}_{n_1,n_2,m}-2}^{\bullet}))\otimes_{\Lambda_{n_1,n_2}}\Lambda_{I_{m}}
\]
and the morphism $H_{m-1}^\dagger$ of \eqref{Hn}.


\begin{lemma}\label{LemmaTripleClass}
if $\rho=(k_1,k_2)\in \N^2\cap\Lambda_{n_1,n_2}(\Q)$ we have that
\[
\rho^\ast(\Theta_m(s_1,s_2))=(-1)^m\binom{k_3-2}{m+k_2-1}^{-1}t_{k_1,k_2,k_3}(\rho^\ast(s_1),\rho^\ast(s_2)),
\]
where $k_3=k_1+k_2+2m$.
\end{lemma}
\begin{proof}
A laborious but straightforward computation shows that
\begin{eqnarray*}
&&\rho^\ast\left(\bigtriangledown_{{\bf k}_{\tau_0,n_1}}^{m}(s_1)\cdot s_2\right)=\bigtriangledown_{k_1}^{m}(\rho^\ast(s_1))\cdot \rho^\ast(s_2)=\\
&&=(-1)^{m}\binom{k_{3}-2}{m+k_{2}-1}^{-1} t_{k_1,k_2,k_3}(\rho^\ast(s_1),\rho^\ast(s_2))+\bigtriangledown\left(\sum_{i=0}^{m-1}a_i\bigtriangledown_{k_1}^{i}(\rho^\ast(s_1))\bigtriangledown^{m-1-i}_{k_2}(\rho^\ast(s_2))\right),
\end{eqnarray*}
where 
\[
a_i=(-1)^{i+m+1}\binom{k_{3}-2}{m+k_{2}-1}^{-1}\left(\sum_{j=0}^i\binom{m}{j}\binom{M-2}{k_{1}+j-1}\right).
\]
This implies that
\begin{eqnarray*}
&&(-1)^{m}\binom{k_{3}-2}{m+k_{2}-1}^{-1} t_{k_1,k_2,k_3}(\rho^\ast(s_1),\rho^\ast(s_2))=\\
&&\rho^\ast\left(\bigtriangledown_{{\bf k}_{\tau_0,n_1}}^{m}(s_1)\cdot s_2-\bigtriangledown_{{\bf k}_{n_1,n_2,m}}\left(\sum_{i=0}^{m-1}a_i\bigtriangledown_{{\bf k}_{\tau_0,n_1}}^{i}(s_1)\bigtriangledown^{m-1-i}_{{\bf k}_{\tau_0,n_2}}(s_2)\right)\right).
\end{eqnarray*}
Hence the result follows from the fact that the class of $\bigtriangledown_{{\bf k}_{\tau_0,n_1}}^{m}(s_1)\cdot s_2$ in the de Rham cohomology group coincides with the class of $\bigtriangledown_{{\bf k}_{\tau_0,n_1}}^{m}(s_1)\cdot s_2-\bigtriangledown_{{\bf k}_{n_1,n_2,m}}\left(\sum_{i=0}^{m-1}a_i\bigtriangledown_{{\bf k}_{\tau_0,n_1}}^{i}(s_1)\bigtriangledown^{m-1-i}_{{\bf k}_{\tau_0,n_2}}(s_2)\right)$.
\end{proof}

\subsection{$p$-adic interpolation of the operator $\Theta_m$}

We aim to construct the operator $\Theta_{{\bf k}_{\tau_0,n_3}}$ for the universal character ${\bf k}_{\tau_0,n_3}:\Z_p^\times\longrightarrow\Lambda_{n_3}$. A strategy one can think of is to $p$-adically iterate the operators $\Theta_m$ as $m$ varies. Following this strategy, the problem that one immediately finds is that as $m$ grows the ideal $I_m$ grows as well. Since we have to invert $I_m$ in order to define $H_m^\dagger$, we lose control of the denominators.

One can think of the strategy of defining $\Theta_{{\bf k}_{\tau_0,n_3}}$ in the space of $p$-adic modular forms, namely, sections on the ordinary locus. This is the strategy used in \cite{BM}. By equation \eqref{HdRWFil}, we have to restrict ourselves to the case of slope $\leq h$ in order to control the filtration where $\Theta_{{\bf k}_{\tau_0,n_3}}$ lives. Using the interpolation of the Serre operator introduced in \cite{BM}, one could define $\Theta_{{\bf k}_{\tau_0,n_3}}$ provided by the existence of a slope $\leq h$ projector $e_{\leq h}$ on the space of $p$-adic modular forms. But there is no such projector acting on the space of $p$-adic modular forms except for the ordinary case.

The strategy followed by Andreatta and Iovita in \cite{AI17} relies on p-adically interpolate the operator $\bigtriangledown_{{\bf k}_{\tau_0,n}}^m$ acting on $\W':=\bigoplus_k\W^{{\bf k}_{\tau_0,n}+2k}\subset (g_n\circ f_0)_\ast\cO_{\V_0(\cH_0^\sharp,s_0)}\otimes g_{1,\ast}\cO_{\mathfrak{IG}_1}\hat\otimes\Lambda_n$. Since on the space of nearly overconvergent modular forms we have the projector $e_{\leq h}$, we have the operator $H^{\dagger,\leq h}$ of the diagram \eqref{Hdh}, hence we can proceed with a well defined construction of $\Theta_{{\bf k}_{\tau_0,n_3}}$  on the space of nearly overconvergent modular forms.
In order to ensure the convergence of the series involved in the definition of $\bigtriangledown_{{\bf k}_{\tau_0,n}}^{\bf s}$, for a universal character ${\bf s}$, one has to control de valuation of the corresponding terms, and this can be done by computing the Gauss-Manin connection on the ordinary locus. There we can use Serre-Tate coordinates and the corresponding $q$-expansions.


\subsection{Twists by finite characters}
This section is analogous to \cite[\S 3.8]{AI17}. As always $r\geq n$ and let $\cG_0$ be the universal formal group over $\mathcal{IG}_{2n,r+n}$. Let $C_n\subset \cG_0[p^n]$ be the canonical subgroup and let us consider 
\[
\pi:\cG_0\longrightarrow \cG_0':=\cG_0/C_n\qquad C_n':=C_{2n}/C_n.
\]
It is clear that $C_n'$ is the canonical subgroup of $\cG_0'$ and the dual isogeny $\lambda:\cG_0'\rightarrow\cG_0$ maps $C_n'$ to $C_n$. This construction defines a morphism $t:\mathcal{IG}_{2n,r+n}\rightarrow\mathcal{IG}_{n,r}$, since a trivialization of $D_{2n}$ induces by multiplication by $p^n$ a trivialization of $D_n'=(C_n')^\vee$. 

If $C_n'':=\ker{\lambda}$, we have that $\cG_0[p^n]=C_n'\times C_n''$ as group schemes over $\mathcal{IG}_{2n,r+n}$. The universal trivialization provides isomorphisms
\[
s:\Z/p^n\Z\longrightarrow C_n'',\qquad s^\vee:C_n'\longrightarrow \mu_{p^n},
\]
since the Weil pairing and the polarization identifies $C_n'\simeq (C_n'')^\vee$. 
 
A choice of a primitive $p^n$-root of unity $\xi$ provides an identification $\Z/p^n\Z\simeq{\rm Hom}(\Z/p^n\Z,\mu_{p^n})$, $j\mapsto(1\mapsto\xi^j)$. Thus, we obtain a bijection
\[
{\rm Hom}(C_n'',C_n')\longrightarrow{\rm Hom}(\Z/p^n\Z,\mu_{p^n})\simeq \Z/p^n\Z\,\qquad g\longmapsto s^\vee\circ g\circ s.
\]
Write $\rho_j\in {\rm Hom}(C_n'',C_n')$ for the morphism corresponding to $j\in \Z/p^n\Z$, and let us consider 
$t_j:\mathcal{IG}_{2n,r+n}\rightarrow\mathcal{IG}_{n,r}$ for the morphism induced by making quotient by $(\rho_j\times{\rm Id})(C_n'')\subset\cG_0'[p^n]$, with the induced trivialization of $(C_{n,j})^\vee$, where $C_{n,j}$ is the image of $C_n'$ in $\cG_{0}^j:=\cG_0/ (\rho_j\times{\rm Id})(C_n'')$ which coincides with the canonical subgroup. Write also $t_j:\mathfrak{IG}_{2n,r+n}\rightarrow\mathfrak{IG}_{n,r}$ for the corresponding morphism taking normalizations. Over $\mathfrak{IG}_{2n,r+n}$ we have isogenies on the universal formal groups
\[
\cG_0\stackrel{\lambda}{\longleftarrow} \cG_0'\stackrel{\lambda_j}{\longrightarrow} \cG_0^j
\]
respecting canonical groups and trivializations. By functoriality, this gives rise to the corresponding maps
\[
\cH_0^\sharp\stackrel{\lambda^\sharp}{\longleftarrow} (\cH_0')^\sharp\stackrel{\lambda_j^\sharp}{\longrightarrow} (\cH_0^j)^\sharp.
\]
By \cite[Lemma 6.4]{AI17} the images coincide ${\rm Im}\lambda^\sharp={\rm Im}\lambda_j^\sharp$, hence it gives rise to a isomorphism $f_j: (\cH_0^j)^\sharp\rightarrow  \cH_0^\sharp$ preserving Hodge exact sequences and Gauss-Manin connections. Again by functoriality, this gives rise to a morphism over $\mathfrak{IG}_{2n,r+n}$
\[
f_j^\ast: t_j^{\ast}(f_0)_\ast\cO_{\V_Q(\cH_0^\sharp,s_0)}\longrightarrow (f_0)_\ast\cO_{\V_Q(\cH_0^\sharp,s_0)}
\] 
preserving filfrations and Gauss-Manin connections.
Thus, after extending scalars in order to have the root $\xi$, we obtain a morphism for any finite character $\chi:(\Z/p^n\Z)^\times\rightarrow \bar\Q_p^\times$
\[
\theta_\chi=g_{\chi^{-1}}^{-1}\sum_j\chi(j)^{-1}f_j^\ast\circ t_j^\ast:H^0\left(\dX_r,(g_n\circ f_0)_\ast\cO_{\V_Q(\cH_0^\sharp,s_0)}\right)\longrightarrow H^0\left(\dX_{r+n},(g_{2n}\circ f_0)_\ast\cO_{\V_Q(\cH_0^\sharp,s_0)}\right),
\]
where $g_{\chi^{-1}}:=\sum_n\chi(n)^{-1}\xi^n$ is a Gauss sum.

The following result is completely analogous to \cite[Lemma 3.31]{AI17}:
\begin{lemma}\label{fintwist}
The morphism $\theta_\chi$ provides a morphism 
\[
\bigtriangledown_{{\bf k}_{\tau_0,n}}^{\chi}:H^0\left(\dX_r\times\dW_n,\W^{{\bf k}_{\tau_0,n}}\right)\longrightarrow H^0\left(\dX_{r+n}\times\dW_{n},\W^{{\bf k}_{\tau_0,n}+2\chi}\right).
\]
\end{lemma}

\subsection{$q$-expansions on unitary Shimura curves}

Over the ordinary locus $\dX_{\rm ord}$ it is constructed in \cite{BM} an affine $\Spf(W)$-formal scheme $\dX(\infty)$, where $W$ is the ring of Witt vectors of $\bar F_p$. $\dX(\infty)$ represents ordinary abelian varieties   
$(A,\iota,\theta,\alpha^{\dP_0})$ together with frames: 
\begin{equation}\label{bigframe}
0 \longrightarrow\G_m[p^\infty] \stackrel{\imath}{\longrightarrow} A[p^\infty]^{-,1} \stackrel{\pi}{\longrightarrow}\Q_p/\Z_p\times\prod_{\dP\neq\dP_0}(F_{\dP}/\cO_\dP)^2\longrightarrow 0.
\end{equation}
Moreover, $\dX(\infty)$ is equipped with a classifying morphism $c:\dX(\infty)/W\rightarrow\G_m/W$. The local coordinates given by the morphism $c$ are called \emph{Serre-Tate coordinates}.

As explained in \cite[\S 8.2]{BM}, given a point $P\in \dX(\infty)$ providing a frame
\begin{equation}\label{smallframe} 
0 \longrightarrow\G_m[p^\infty] \stackrel{\imath}{\longrightarrow} \cG_0[p^\infty] \stackrel{\pi}{\longrightarrow}\Q_p/\Z_p\longrightarrow 0
\end{equation}
by restriction of \eqref{bigframe}, its image is $c(P)=q\in\G_m$ such that
\[
\cG_0\simeq E_q:=(\G_{m,R}[p^{\infty}]\oplus\Q_p)\slash\langle((1+q)^m,-m),\;m\in\Z_p\rangle,\qquad 0\longrightarrow\G_{m,R}[p^\infty]\stackrel{\imath}{\longrightarrow} E_q\stackrel{\pi_q}{\longrightarrow}\Q_p/\Z_p\longrightarrow 0,
\]
where $\imath (a)=(a,0)$ and $\pi_q(a,b)=b$. Moreover, we have that $C_n=\{(\xi^n,0),\;n\in\Z/p^n\Z\}$. This implies that
\[
\cG_0'=E_{(q+1)^{p^n}-1}:=(\G_{m,R}[p^{\infty}]\oplus\Q_p)\slash\langle((1+q)^{p^nm},-m),\;m\in\Z_p\rangle,
\]
and the morphism $\lambda:\cG_0'\rightarrow\cG_0$ is given by the morphism
\[
\lambda:E_{(q+1)^{p^n}-1}\longrightarrow E_q;\qquad \lambda(a,x)=(a,p^n x),\qquad C_n''=\ker(\lambda)=\left\{\left((1+q)^m,-\frac{m}{p^n}\right),\;m\in\Z_p\right\}.
\]
For any $j\in(\Z/p^n\Z)^\times$, the morphism $\rho_j:C_n''\rightarrow C_n'=\{(\xi^m,0),m\in\Z/p^n\Z\}$ is given by
\[
\rho_j:C_n''\longrightarrow C_n', \; \left((1+q)^m,-\frac{m}{p^n}\right)\longmapsto (\xi^{mj},0),\qquad (\rho_j\times{\rm Id})C_n''=\left\{\left(\xi^{mj}(1+q)^m,-\frac{m}{p^n}\right),\;m\in\Z_p\right\}.
\]
This implies that (see \cite[\S 8.2]{BM}) 
\[
\cG_0^j= E_{\xi^j(q+1)-1}=(\G_{m,R}[p^{\infty}]\oplus\Q_p)\slash\langle(\xi^{mj}(1+q)^{m},-m),\;m\in\Z_p\rangle.
\]
Thus, for any $f(q)\in W[[q]]=\cO_{\G_m}$, we have $t_j^\ast f(q)=f(\xi^j(q+1)-1)$.

Over $\dX(\infty)$, the sheaf $\cH_{0}$ has a cannonical basis $\omega,\eta$ such that $\omega$ is a basis of $\omega_0$. In Serre-Tate coordinates the Gauss-Manin connection acts as follows (see \cite[proof Theorem 9.9]{BM}): Given $f_1(q),f_2(q)\in W[[q]]=\cO_{\G_m}$,
\begin{equation}\label{actGM}
\bigtriangledown(f_1(q)\omega+f_2(q)\eta)=\partial f_1\omega^3+\left(f_1+\partial f_2\right)\eta\omega^2,\qquad\partial f_i:=(q+1)\frac{df_i}{dq}.
\end{equation}
By the local description of $\W_0^{{\bf k}_{\tau_0,n}}$, for an small open subset $U\subset \dX(\infty)$, the space $\W_0^{{\bf k}_{\tau_0,n}}(U)$ is generated by elements of the form
\[
V_{{\bf k}_{\tau_0,n},m}:=P_n(Z)\frac{T^m}{(1+p^nZ)^m}\in{\rm Fil}_m\W_0^{{\bf k}_{\tau_0,n}},\qquad P_n(Z)=\sum_{i\geq 0}\binom{w_n}{i}p^{ni}Z^i,
\] 
where $(1+p^nZ)$ corresponds to $\omega$ and $p^nT$ corresponds to $\eta$. By functoriality, the morphism $f_j^\ast\circ t_j^\ast$ leaves $Z$ and $T$ invariant. Hence, for any finite character $\chi$ of conductor at most $p^n$ and any $f(q)=\sum_na_n(q+1)^n$, we obtain 
\begin{eqnarray*}
\theta_\chi(f(q)V_{{\bf k}_{\tau_0,n},m})&=&g_{\chi^{-1}}^{-1}\sum_j\chi(j)^{-1}f_j^\ast\circ t_j^\ast(f(q)V_{{\bf k}_{\tau_0,n},m})=g_{\chi^{-1}}^{-1}\sum_j\chi(j)^{-1}f(\xi^j(q+1)-1)V_{{\bf k}_{\tau_0,n},m}\\
&=&g_{\chi^{-1}}^{-1}\sum_j\chi(j)^{-1}\left(\sum_na_n\xi^{jn}(1+q)^n\right)V_{{\bf k}_{\tau_0,n},m}=f_\chi(q)V_{{\bf k}_{\tau_0,n},m},
\end{eqnarray*}
where $f_{\chi}(q):=\sum_n\chi(n)a_n(q+1)^n$ (here $\chi(n)=0$ if $\gcd(n,p)\neq 1$).
\begin{rmk}
Notice that the image of $\theta_\chi$ lies in the kernel of $U_{\dP_0}$ by \cite[\S 8.3]{BM}.
\end{rmk}

\subsection{Iterations of the Gauss-Manin operator}

For any character $(k:\Z_p^\times\longrightarrow R)\in\dW_{\tau_0,n}(R)$, let $\W_0^{k}\subset(g_n\circ f_0)_\ast\cO_{\V_Q(\cH_0^\sharp,s_0)}\hat\otimes R$ be the specialization of $\W_0^{{\bf k}_{\tau_0,n}}$.

Let $\Lambda_{\tau_0,n}'$ be isomorphic to $\Lambda_{\tau_0,n}$.
We aim to define $\bigtriangledown_{{\bf k}_{\tau_0,n}}^{{\bf s}_{\tau_0,n}}: \W_0^{{\bf k}_{\tau_0,n}}\rightarrow \W_0^{{\bf k}_{\tau_0,n}+2{\bf s}_{\tau_0,n}}$, where ${\bf s}_{\tau_0,n}:\Z_p^\times\rightarrow\Lambda_{\tau_0,n}'$ is the universal character seen in $\dW_{\tau_0,2n}(\Lambda_{\tau_0,n}\hat\otimes\Lambda_{\tau_0,n}')$.
Recall that by \eqref{univK}
\[
{\bf s}_{\tau_0,n}(\beta)=\sum_{\gamma}{\bf s}(\gamma)\cdot 1_{\gamma+p^{n}\Z_p}(\beta)\cdot\sum_{i\geq 0}\binom{w_n'}{i}\left(\frac{\beta-\gamma}{\gamma}\right)^i,\quad \gamma=w(s)\exp(jp),\; 0\leq j\leq p^{n-1}-1, \;s\in(\Z/p\Z)^\times
\]
where  $w_n'=p^{-1}\log(1+T)\in p^{1-n}\Lambda_{\tau_0,n}'$, $w$ is the Teichm\"uller character and ${\bf s}(w(s)\exp(jp))=(1+T)^j(s)\in\Lambda_{\tau_0}=\Z_p[(\Z/p\Z)^\times][[T]]$.

This implies that $\bigtriangledown_{{\bf k}_{\tau_0,n}}^{{\bf s}_{\tau_0,n}}$ should be a limit of the form
\begin{equation}\label{GMs}
\bigtriangledown_{{\bf k}_{\tau_0,n}}^{{\bf s}_{\tau_0,n}}=\sum_{\gamma}{\bf s}(\gamma)\cdot\sum_{i\geq 0}\binom{w_n'}{i}\left(\frac{\bigtriangledown_{{\bf k}_{\tau_0,n}}-\gamma {\rm Id}}{\gamma}\right)^i\circ \bigtriangledown_{{\bf k}_{\tau_0,n}}^{\gamma+p^{n}\Z_p}.
\end{equation}
First, using the definition of the operator $\bigtriangledown_{{\bf k}_{\tau_0,n}}^{\chi}$ of Lemma \ref{fintwist} we define
\[
\bigtriangledown_{{\bf k}_{\tau_0,n}}^{\gamma+p^{n}\Z_p}:=\frac{1}{\varphi(p^n)}\sum_{\chi}\chi(\gamma)^{-1}\bigtriangledown_{{\bf k}_{\tau_0,n}}^{\chi},
\]
where the sum is taken over all characters of conductor $p^n$ and $\varphi(p^n)=(p-1)p^{n-1}$.
Notice that, locally over the ordinary locus and in Serre-Tate coordinates
\begin{eqnarray*}
\bigtriangledown_{{\bf k}_{\tau_0,n}}^{\gamma+p^{n}\Z_p}(a_n(1+q)^nV_{{\bf k}_{\tau_0,n},m})&=&a_n(1+q)^nV_{{\bf k}_{\tau_0,n},m}\frac{1}{\varphi(p^n)}\sum_{\chi}\chi(\gamma)^{-1}\chi(n)\\
&=&\left\{\begin{array}{lc}a_n(q+1)^nV_{{\bf k}_{\tau_0,n},m},&n\equiv\gamma \;({\rm mod}\;p^n)\\0,&n\not\equiv\gamma \;({\rm mod}\;p^n).\end{array}\right.
\end{eqnarray*}

Once $\bigtriangledown_{{\bf k}_{\tau_0,n}}^{\gamma+p^{n}\Z_p}$ is defined, the existence of $\bigtriangledown_{{\bf k}_{\tau_0,n}}^{{\bf s}_{\tau_0,n}}$ will follow from the convergence of the series given in \eqref{GMs}.

Assuming as always that $r\geq n$, after applying $\bigtriangledown_{{\bf k}_{\tau_0,n}}^{\gamma+p^{n}\Z_p}$ to a section in $\dX_{r}\times\dW_n$, we obtain a section in $\dX_{r+n}\times\dW_n$. Over $\dX_{r+n}\times\dW_n$ we will denote by $\W_0^{{\bf k}_{\tau_0,n}+2{\bf s}_{\tau_0,n}}$ the specialization of $\W_0^{{\bf k}_{\tau_0,2n}}$ at ${\bf k}_{\tau_0,n}+2{\bf s}_{\tau_0,n}$. Since any locally analytic function with radius of analycity $p^n$ has also radius of analicity $p^{2n}$, we can always see a section in $\W_0^{{\bf k}_{\tau_0,n}}$ as an specialization of a section in $\W_0^{{\bf k}_{\tau_0,2n}}$.

To ensure the convergence of the series \eqref{GMs}, we will prove in Proposition \ref{pdiviter} that 
\[
\sum_{\gamma}{\bf s}(\gamma)\cdot\binom{w_n'}{i}\left(\frac{\bigtriangledown_{{\bf k}_{\tau_0,n}}-\gamma {\rm Id}}{\gamma}\right)^i\circ \bigtriangledown_{{\bf k}_{\tau_0,n}}^{\gamma+p^{n}\Z_p}(s)\in H^0(\dX_{\rm ord}\times\dW_{n}\times\dW_{\tau_0,n}',p^{i\frac{p}{p-1}}(g_{2n}\circ f_0)_\ast\cO_{\V_Q(\cH^\sharp,s_0)}),
\]
for any $s\in H^0(\dX_r\times\dW_n,\W_0^{{\bf k}_{\tau_0,n}})$ with $\dW_{\tau_0,n}':={\rm Spf}(\Lambda_{\tau_0,n}')$. By Theorem \ref{GMthm} we have
\[
A_i:={\rm Hdg}^{ic_n}\sum_{\gamma}{\bf s}(\gamma)\cdot\binom{w_n'}{i}\left(\frac{\bigtriangledown_{{\bf k}_{\tau_0,n}}-\gamma {\rm Id}}{\gamma}\right)^i\circ \bigtriangledown_{{\bf k}_{\tau_0,n}}^{\gamma+p^{n}\Z_p}\in H^0(\dX_{r+n}\times\dW_{n}\times\dW_{\tau_0,n}',(g_{2n}\circ f_0)_\ast\cO_{\V_Q(\cH^\sharp,s_0)}).
\]
Hence applying \cite[Proposition 4.10]{AI17}, we deduce that there exists $b$ depending on $r$ and $n$ such that 
\[
A_i(s)\in p^{[i/2]} H^0(\dX_{b}\times\dW_{n}\times\dW_{\tau_0,n}',(g_{2n}\circ f_0)_\ast\cO_{\V_Q(\cH^\sharp,s_0)}),
\]
for any $s\in H^0(\dX_r\times\dW_n,\W_0^{{\bf k}_{\tau_0,n}})$. This provides the existence of the operator
\[
\bigtriangledown_{{\bf k}_{\tau_0,n}}^{{\bf s}_{\tau_0,n}}:=\sum_{i\geq 0}A_i:H^0(\dX_r\times\dW_n,\W_0^{{\bf k}_{\tau_0,n}})\longrightarrow H^0(\dX_b\times\dW_n\times\dW_{\tau_0,n}',\W_0^{{\bf k}_{\tau_0,n}+2{\bf s}_{\tau_0,n}})^{U_{\dP_0}=0}.
\]
Moreover, by the action of $U_{\dP_0}$ and $V_{\dP_0}$ on $q$-expansions described in \cite[\S 8.3]{BM}, for any integer $m\in\dW_{\tau_0,n}'(\C_p)\cap \N$ and any section $s\in H^0(\dX_r\times\dW_n,\W_0^{{\bf k}_{\tau_0,n}})$ we have 
\begin{equation}\label{GMesp}
m^\ast\bigtriangledown_{{\bf k}_{\tau_0,n}}^{{\bf s}_{\tau_0,n}}(s)=\bigtriangledown_{{\bf k}_{\tau_0,n}}^{m}(s^{[p]})=\left(\bigtriangledown_{{\bf k}_{\tau_0,n}}^{m}\circ(1-V_{\dP_0}\circ U_{\dP_0})\right)(s).
\end{equation}

\begin{rmk}
The existence of $\bigtriangledown_{{\bf k}_{\tau_0,n}}^{{\bf s}_{\tau_0,n}}$ generalize the results of Andreatta-Iovita in \cite{AI17} even for $F=\Q$. Indeed, they were only able to construct $\bigtriangledown_{{\bf k}_{n}}^{s}$ under certain analycity conditions for $s$ (\cite[Assumption 5.4]{AI17}). The innovative tools that have allow us to construct $\bigtriangledown_{{\bf k}_{\tau_0,n}}^{{\bf s}_{\tau_0,n}}$ for the universal character ${\bf s}_{\tau_0,n}$ and were not available in \cite{AI17} are the locally analytic description of ${\bf s}_{\tau_0,n}$, provided by Equation \eqref{univK}, and the construction of $\W_0^{{\bf k}_{\tau_0,n}}$ by means of the new formal vector bundles $\V_Q(\mathcal{E},s_1,\cdots, s_m)$ described in \S \ref{ss:formal vector bundles}, that ensure the convergence of the power series $A_i$ as we will see in the next section. 
\end{rmk}


\subsection{Computations in Serre-Tate coordinates}

For any $a\in\N$, the local description of $\W_0^{{\bf k}_{\tau_0,n}+a}$ over $\dX(\infty)$ tells us that any section of $\W_0^{{\bf k}_{\tau_0,n}+a}(U)$, for some open $U\subset \dX(\infty)$, is generated by elements of the form 
\[
V_{{\bf k}_{\tau_0,n}+a,m}:=P_{2n}(Z)(1+p^{2n}Z)^{a}\frac{T^m}{(1+p^{2n}Z)^m}.
\] 
Applying the relations of \eqref{actGM}, we obtain that
\[
\bigtriangledown_{{\bf k}_{\tau_0,n}+a}(V_{{\bf k}_{\tau_0,n}+a,m})=p^{2n}(w_n+r-m)\cdot V_{{\bf k}_{\tau_0,n}+a+2,m+1}, 
\]
and this implies that (see \cite[Lemma 3.38]{AI17})
\begin{equation}\label{GMsclas}
\bigtriangledown_{{\bf k}_{\tau_0,n}}^s(f\cdot V_{{\bf k}_{\tau_0,n},m})=\sum_{i=0}^sp^{2ni}\binom{s}{i}\binom{w_n+s-m-1}{i}i!\cdot \partial^{s-i} f\cdot V_{{\bf k}_{\tau_0,n}+2s,m+i},\qquad s\in\N.
\end{equation}

Assume that $f(q)=a_n(1+q)^n$, where $p\nmid n$, and assume that $n\equiv\gamma_0$ modulo $p^n$. Since in this case $\partial f=nf$, we compute that
\begin{eqnarray*}
&&\sum_{\gamma}\binom{w_n'}{i}\left(\frac{\bigtriangledown_{{\bf k}_{\tau_0,n}}-\gamma {\rm Id}}{\gamma}\right)^i\circ \bigtriangledown_{{\bf k}_{\tau_0,n}}^{\gamma+p^{n}\Z_p}(f\cdot V_{{\bf k}_{\tau_0,n},m})=\binom{w_n'}{i}\left(\frac{\bigtriangledown_{{\bf k}_{\tau_0,n}}-\gamma_0 {\rm Id}}{\gamma_0}\right)^i(f\cdot V_{{\bf k}_{\tau_0,n},m})\\
&=&\binom{w_n'}{i}\sum_{j=0}^i(-1)^{i-j}\binom{i}{j}\gamma_0^{-j}\bigtriangledown_{{\bf k}_{\tau_0,n}}^j(f\cdot V_{{\bf k}_{\tau_0,n},m})\\
&=&\binom{w_n'}{i}\sum_{j=0}^i\binom{i}{j}(-1)^{i-j}\gamma_0^{-j}\sum_{k=0}^jp^{2nk}\binom{j}{k}\binom{w_n+j-m-1}{k}k!\cdot \partial^{j-k} f\cdot V_{{\bf k}_{\tau_0,n}+2j,m+k}\\
&=&\binom{w_n'}{i}fV_{{\bf k}_{\tau_0,n},m}\sum_{j=0}^i\sum_{k=0}^{j}\binom{i}{j}\binom{j}{k}\binom{M+j}{k}k!(-1)^{i-j} A^{k}B^j,
\end{eqnarray*}
where $M:=w_n-m-1$ $A:=\left(\frac{p^{2n}T}{n(1+p^{2n}Z)}\right)$ and $B:=\left(\frac{n(1+p^{2n}Z)^{2}}{\gamma_0}\right)$

\begin{lemma}
We have that 
\begin{eqnarray*}
\sum_{j=0}^i\sum_{k=0}^{j}\binom{i}{j}\binom{j}{k}\binom{M+j}{k}k!(-1)^{i-j} A^{k}B^j=
\\=(-1)^i\sum_{k=0}^i\binom{i}{k}(-AB)^kk!\sum_{j=0}^k\binom{M+k}{k-j}\binom{i-k}{j}(1-B)^{i-k-j}(-B)^j.
\end{eqnarray*}
\end{lemma}
\begin{proof}
We compute that
\begin{eqnarray*}
\sum_{j=0}^k\binom{M+k}{k-j}\binom{i-k}{j}(1-B)^{i-k-j}(-B)^j&=&\sum_{j=0}^k\sum_{r=0}^{i-k-j}\binom{M+k}{k-j}\binom{i-k}{j}\binom{i-k-j}{r}(-B)^{r+j}=\\
=\sum_{j=0}^{i}\sum_{s=j}^{i-k}\binom{M+k}{k-j}\binom{i-k}{j}\binom{i-k-j}{s-j}(-B)^{s}&=&\sum_{s=0}^{i-k}\binom{i-k}{s}(-B)^{s}\sum_{j=0}^{k}\binom{M+k}{k-j}\binom{s}{j}=\\
&=&\sum_{s=0}^{i-k}\binom{i-k}{s}\binom{M+k+s}{k}(-B)^{s}.
\end{eqnarray*}
Thus, we obtain
\begin{eqnarray*}
&&(-1)^i\sum_{k=0}^i\binom{i}{k}(-AB)^kk!\sum_{j=0}^k\binom{M+k}{k-j}\binom{i-k}{j}(1-B)^{i-k-j}(-B)^j=\\
&=&(-1)^i\sum_{k=0}^i\sum_{s=0}^{i-k}\binom{i}{k}\binom{i-k}{s}\binom{M+k+s}{k}k!A^k(-B)^{k+s}=\\
&=&(-1)^i\sum_{k=0}^i\sum_{j=k}^{i}\binom{i}{k}\binom{i-k}{j-k}\binom{M+j}{k}k!A^k(-B)^{j},
\end{eqnarray*}
hence the result follows from a change of variable and an easy equality with binomials.
\end{proof}
Since $p^{2n}\mid A$, $p^n\mid (1-B)$ and $k!\binom{M+k}{m}\in p^{-nm}\Lambda_{\tau_0,n}$ (see Remark \ref{valbinom}), we deduce that
\begin{eqnarray*}
&&p^{n(i-2k)}\mid k!\binom{M+k}{k-j}\binom{i-k}{j}(1-B)^{i-k-j}(-B)^j,\\
&&p^{ni}\mid  (-1)^i\sum_{k=0}^i\binom{i}{k}(-AB)^kk!\sum_{j=0}^k\binom{M+k}{k-j}\binom{i-k}{j}(1-B)^{i-k-j}(-B)^j.
\end{eqnarray*}
By Remark \ref{valbinom} $\binom{w_n'}{i}\in p^{-i\left(n-\frac{p}{p-1}\right)}\Lambda_{\tau_0,n}'$, hence the above lemma provides the following result:
\begin{prop}\label{pdiviter}
For any $s\in H^0(\dX_{\rm ord},\W_0^{{\bf k}_{\tau_0,n}})$,
\[
\sum_{\gamma}\binom{w_n'}{i}\left(\frac{\bigtriangledown_{{\bf k}_{\tau_0,n}}-\gamma {\rm Id}}{\gamma}\right)^i\circ \bigtriangledown_{{\bf k}_{\tau_0,n}}^{\gamma+p^{n}\Z_p}(s)\in p^{i\frac{p}{p-1}}H^0(\dX_{\rm ord}\times\dW_{n}\times\dW_{\tau_0,n}',(g_{2n}\circ f_0)_\ast\cO_{\V_Q(\cH^\sharp,s_0)}).
\]
\end{prop}
\begin{rmk}
This computation of the $p$-adic valuation shows the necessity of working with radius of analicity $p^{2n}$. One could predict this fact by observing \eqref{GMsclas} and guessing that over $\dX_{\rm ord}$
\[
\bigtriangledown_{{\bf k}_{\tau_0,n}}^{{\bf s}_{\tau_0,n}}(f\cdot V_{{\bf k}_{\tau_0,n},m})=\sum_{i\geq 0}p^{2ni}\binom{w_n'}{i}\binom{w_n+w_n'-m-1}{i}i!\cdot \partial^{{\bf s}_{\tau_0,n}-i} f\cdot V_{{\bf k}_{\tau_0,n}+2{\bf s}_{\tau_0,n},m+i},
\]
where $\partial^{{\bf s}_{\tau_0,n}-i}a_n(1+q)^n={\bf s}_{\tau_0,n}(n)a_n(1+q)^n$, if $p\nmid n$, and $\partial^{{\bf s}_{\tau_0,n}-i}a_n(1+q)^n=0$, otherwise.
\end{rmk}

\subsection{Triple products for any weight}

The existence of $\bigtriangledown_{{\bf k}_{\tau_0,n}}^{{\bf s}_{\tau_0,n}}$ together with the morphism
\[
H^{\dagger,\leq h}:H^0(\cX_r\times\cW_n,\W^{{\bf k}_{\tau_0,n}})\longrightarrow H^0(\cX_r\times\cW_n,\Omega_0^{{\bf k}_{\tau_0,n}})^{\leq h}\otimes_{\Lambda_{\tau_0,n}}\Lambda_{I_m},
\]
of \eqref{Hdh}
provides the triple product
\begin{eqnarray*}
\Theta_{{\bf k}_{\tau_0,n_3}}:H^0(\cX_{r,n_1},\Omega_0^{{\bf k}_{\tau_0,n_1}})\times H^0(\cX_{r,n_2},\Omega_0^{{\bf k}_{\tau_0,n_2}})&\longrightarrow& H^0(\cX_{b,n_1,n_2,n_3},\Omega_0^{{\bf k}_{\tau_0,n_3}^\ast})^{\leq h}\otimes_{\Lambda_{n_1,n_2,n_3}}\Lambda_{I_m}\\
{\bf t}_{\tau_0}(s_1,s_2)&:=&H^{\dagger,\leq h}\left(\bigtriangledown_{{\bf k}_{\tau_0,n_1}}^{{\bf k}_{\tau_0,n_3}}s_1\cdot s_2\right),
\end{eqnarray*}
where ${\bf k}_{\tau_0,n_3}^\ast:={\bf k}_{\tau_0,n_1}+{\bf k}_{\tau_0,n_2}+2{\bf k}_{\tau_0,n_3}$, $\cX_{r,n_i}:=\cX_r\times\cW_{n_i}$, $\cX_{b,n_1,n_2,n_3}:=\cX_b\times\cW_{n_1}\times\cW_{n_2}\times\cW_{n_3}$, $\Lambda_{n_1,n_2,n_3}:=\Lambda_{\tau_0,n_1}\otimes\Lambda_{\tau_0,n_2}\otimes\Lambda_{\tau_0,n_3}$, and the superindex $\leq h$ relies on the slope  of the operator $U_{\dP_0}$.
By Equations \eqref{GMesp} and \eqref{Hdh} we have that, for any $m\in \cW_{\tau_0,n_3}(\C_p)\cap\N$,
\begin{equation}\label{espwtGM}
m^\ast \Theta_{{\bf k}_{\tau_0,n_3}}(s_1,s_2)=e_{\leq h}\Theta_m(s_1^{[p]},s_2).
\end{equation}
Recall that sections on the sheaf $\Omega^{{\bf k}_n}$ can be described using Katz interpretation (\S \ref{ss:overconvergent modular forms a la katz}) as rules that assign to any tuple $(A,\imath,\theta,\alpha^{\dP_0},w)$ over $R$ a locally analytic distribution $\mu(A,\imath,\theta,\alpha^{\dP_0},w)\in D_n^{{\bf k}_n^{\tau_0}}(\cO^{\tau_0},R)$. Recall also that $w=(w_0,w^{\tau_0})$ is a basis for $\Omega=\Omega_0\oplus\bigoplus_{\dP\neq\dP_0,\tau}\omega_{\dP,\tau}$.

We denote by $\Lambda_F^1$ the $(d+1)$-dimesional weight space $\Lambda_{F}^1:=\Z_p[[\cO^\times\times\Z_p^\times]]$. Similarly, we define 
\[
\Lambda_n^1:=\Lambda_F^1\left\langle\frac{T_1^{p^n}}{p},\cdots,\frac{T_d^{p^n}}{p},\frac{T^{p^n}}{p}\right\rangle,\qquad \dW^1:={\rm Spf}(\Lambda_F^1),\qquad \dW_n^1:={\rm Spf}(\Lambda_n^1).
\]
The morphism $\cO^\times\rightarrow \cO^\times\times\Z_p^\times$ given by $t\mapsto (t^{-2},{\rm N}(t))$, where ${\rm N}(t)$ is the norm of $t$, provides a morphism of weight spaces $k:\dW^1\rightarrow\dW$. By \cite[Lemma 7.1]{BM} it gives rise to a morphism $k:\dW_n^1\rightarrow \dW_n$.

We fix integers $r\geq n_3\geq n_1,n_2\geq 1$. For $i= 1, 2, 3$ we denote by $({\bf r}_{n_i},{\bf \nu}_{n_i}):\cO^\times\times\Z_p^\times\rightarrow\Lambda_{n_i}^{1}$ the universal characters of $\dW^1_{n_i}$. Then we put ${\bf k}_{n_3}:= k({\bf r}_{n_3},\nu_{n_1}+ \nu_{n_2})$ and ${\bf k}_{n_i}:= k({\bf r}_{n_i},{\bf \nu}_{n_i})$ for $i= 1, 2$. 

 We put $\cR:=\Lambda_{n_1}^1\hat\otimes\Lambda_{n_2}^1\hat\otimes_{{\bf \nu}_{n_3}=\nu_{n_1}+{\bf \nu}_{n_2}} \Lambda_{n_3}^1\simeq \Lambda_{n_1}^1\hat\otimes\Lambda_{n_2}^1\hat\otimes \Lambda_{n_3}$ and consider the characters:
\begin{eqnarray*}
{\bf m}_1^{\tau_0}:={\bf r}_{n_1}^{\tau_0}-{\bf r}_{n_3}^{\tau_0}-{\bf r}^{\tau_0}_{n_2}+{\bf \nu}_{n_2}\circ N:\cO^{\tau_0\times}&\longrightarrow&\cR,\\
{\bf m}_2^{\tau_0}:={\bf r}^{\tau_0}_{n_2}-{\bf r}^{\tau_0}_{n_1}-{\bf r}^{\tau_0}_{n_3}+{\bf \nu}_{n_1}\circ N:\cO^{\tau_0\times}&\longrightarrow&\cR,\\
{\bf m}_3^{\tau_0}:={\bf r}^{\tau_0}_{n_3}-{\bf r}^{\tau_0}_{n_1}-{\bf r}^{\tau_0}_{n_2}:\cO^{\tau_0\times}&\longrightarrow&\cR,\\
{\bf m}_{3,\tau_0}:={\bf r}_{n_1,\tau_0}+{\bf r}_{n_2,\tau_0}-{\bf r}_{n_3,\tau_0}:\Z_p^\times&\longrightarrow&\cR,
\end{eqnarray*}
where $N:\cO^{\tau_0\times}\rightarrow\Z_p^\times$ denotes the norm map.
In the same way as in \cite[\S 10.2]{BM} 
we denote $\underline{\Delta}^{\tau_0}\in C_{n_1}^{{\bf k}_{n_1}^{\tau_0}}(\cO^{\tau_0},\cR)\otimes_{\cR} C_{n_2}^{{\bf k}_{n_2}^{\tau_0}}(\cO^{\tau_0},\cR)\otimes_{\cR} \bar C_{n_3}^{{\bf k}_{n_3}^{\tau_0}}(\cO^{\tau_0},\cR)$
the function defined by:
\[
\underline{\Delta}^{\tau_0}((x_1,y_1),(x_2,y_2),(x_3,y_3)):={\bf m}_1^{\tau_0}(x_3y_2-x_2y_3) \cdot {\bf m}_2^{\tau_0}(x_3y_1-x_1y_3)\cdot {\bf m}_3^{\tau_0}(x_1y_2-x_2y_1),
\] 
where $\bar C_{n}^{k^{\tau_0}}(\cO^{\tau_0},\cdot)$ denote the $k^{\tau_0}$-homogeneous locally analytic functions on $p\cO^{\tau_0}\times\cO^{\tau_0\times}$, and the function is
extended by 0 where ${\bf m}_3^{\tau_0}$ is not defined.

Let $\mu_i\in H^0(\cX_{r,n_i}^1,\Omega^{{\bf k}_{n_i}})$, where $i=1,2$ and $\cX_{r,n_i}^1:=\cX_{r}\times\cW_{n_i}^1$, be a pair of global sections. The choice of a basis $w^{\tau_0}$ of $\bigoplus_{\dP\neq\dP_0,\tau}\omega_{\dP,\tau}$ in some affine open $U={\rm Spf}(R)$, together with the function $\underline{\Delta}^{\tau_0}$ and the map $\Theta_{{\bf m}_{3, \tau_0}}$ above provide an element $t(\mu_{1},\mu_{2})\in {\rm Hom}_{R}(\bar D_{n_3}^{{\bf k}_{n_3}^{\tau_0}}(\cO^{\tau_0},R),\Omega_0^{{\bf k}_{\tau_0,n_3}}(U))$ as follows:
\[
t(\mu_{1},\mu_{2})(\mu):=\int_{\cO^{\tau_0\times}\times\cO^{\tau_0}}\int_{\cO^{\tau_0\times}\times\cO^{\tau_0}}\int_{p\cO^{\tau_0}\times\cO^{\tau_0\times}}\underline{\Delta}^{\tau_0}(v_1,v_2,v_3)\;d\Theta_{{\bf m}_{3, \tau_0}}(\mu_{1},\mu_{2})(v_1,v_2)d\mu(v_3).
\]
By \cite[\S 10.1]{BM}, the space ${\rm Hom}_{R}(\bar D_{n_3}^{{\bf k}_{n_3}^{\tau_0}}(\cO^{\tau_0},R),\Omega_0^{{\bf k}_{\tau_0,n_3}}(U))$ is endowed with the action of the Hecke operators $U_{\dP}$, for $\dP\neq \dP_0$. These are compact operators, hence we can define the usual finite slope projectors $e_{\leq\underline{h}}$, where $\underline{h}\in\N^{\#\{\dP\neq\dP_0\}}$, onto the spaces ${\rm Hom}_{R}(\bar D_{n_3}^{{\bf k}_{n_3}^{\tau_0}}(\cO^{\tau_0},R),\Omega_0^{{\bf k}_{\tau_0,n_3}}(U))^{\leq\underline{h}}$. The same argument showed in \cite[Lemma 10.2]{BM} implies that we have an isomorphism of $R[U_\dP]_{\dP}$-modules:
\[
{\rm Hom}_{R}(\bar D_{n_3}^{{\bf k}_{n_3}^{\tau_0}}(\cO^{\tau_0},R),\Omega_0^{{\bf k}_{\tau_0,n_3}}(U))^{\leq\underline{h}}\simeq D_{n_3}^{{\bf k}_{n_3}^{\tau_0}}(\cO^{\tau_0},\Omega_0^{{\bf k}_{\tau_0,n_3}}(U))^{\leq\underline{h}}.
\] 
Thus, for any choice of horizontal basis $w^{\tau_0}$ we obtain $e_{\leq\underline{h}}t(\mu_{1},\mu_{2})\in D_{n_3}^{{\bf k}_{n_3}^{\tau_0}}(\cO^{\tau_0},\Omega_0^{{\bf k}_{\tau_0,n_3}}(U))^{\leq\underline{h}}$, satisfying all the properties (B1), (B2), (B3) of \S \ref{locallanaldist}. This implies that we obtain a global section 
\[
t(\mu_1,\mu_2)\in H^0(\cX_{b,n_1,n_2,n_3}^1,\Omega^{{\bf k}_{n_3}})^{\leq \underline{h}},\qquad \cX_{b,n_1,n_2,n_3}^1:=\cX_{b}\times\cW_{n_1}^1\times\cW_{n_2}^1\times\cW_{n_3},
\]
where $\underline{h}\in\N^{\#\{\dP\mid p\}}$ now stands for the finite slopes of the different $U_\dP$.
This construction defines a triple product
\begin{equation}\label{gentripprod}
t:H^0(\cX_{r,n_1}^1,\Omega^{{\bf k}_{n_1}})\times H^0(\cX_{r,n_2}^1,\Omega^{{\bf k}_{n_2}})\longrightarrow H^0(\cX_{b,n_1,n_2,n_3}^1,\Omega^{{\bf k}_{n_3}})^{\leq \underline{h}}\otimes\Lambda_{I_m}.
\end{equation}

\subsection{Triple product $p$-adic L-functions: Finite slope case} 

As in the previous section, let us denote by $({\bf r}_{n},{\bf \nu}_{n}):\cO^\times\times\Z_p^\times\rightarrow\Lambda_{n}^{1}$ the universal character of $\dW^1_{n}$, and we put ${\bf k}_{n}:= k({\bf r}_{n},{\bf \nu}_{n})$. Similarly,  we denote by $({\bf r},{\bf \nu}):\cO^\times\times\Z_p^\times\rightarrow\Lambda_{F}^{1}$ the universal character of $\dW^1_{F}$, and ${\bf k}:= k({\bf r},{\bf \nu})$. In \cite[\S 7.2]{BM} it is described an action of the finite group $\Delta:=(\cO_F)^\times_+/\{u\in(\cO_F)_+^\times,\;u\equiv 1\;{\rm mod}\;\mathfrak{n}\}$ on the space $H^0(\dX_r\times \dW_n^1,\Omega^{{\bf k}_{n}})$: for any $\mu\in H^0(\dX_r\times \dW_n^1,\Omega^{{\bf k}_{n}})$ and $s\in (\cO_F)^\times_+$,
\[
([s]\ast\mu)(A,\imath,\theta,\alpha^{\dP_0},w):={\bf r}_n(s\tau_0(s)^{-2})\cdot \nu_n(\tau_0 s)\cdot \mu(A,\imath,s^{-1}\theta,\gamma_s^{-1}\alpha^{\dP_0},w),
\]
where $\gamma_s\in K_1(\mathfrak{n},p)$ and $\det\gamma_s=s$. Let us consider the space 
\[
M^{r}_{{\bf k}_{n}}(\Gamma_1(\cN,p),\Lambda_{n}^1):=\bigoplus_{\mathfrak{c}\in{\rm Pic}(\cO_F)}H^0(\dX_r^{\mathfrak{c}}\times \dW_n^1,\Omega^{{\bf k}_{n}})^\Delta,
\]
where $\dX_r^{\mathfrak{c}}$ is the overconvergent neighborhood of the unitary Shimura curve associated with $\mathfrak{c}$ defined as in \cite[Remark 4.3]{BM}. Then the $\Lambda_F^1$-modules 
\[
M^{\dagger}_{{\bf k}}(\Gamma_1(\cN,p),\Lambda_F^1):=\varprojlim_{r,n} M^{r}_{{\bf k}_{n}}(\Gamma_1(\cN,p),\Lambda_{n}^1),\qquad M^{\dagger}_{{\bf k}}(\Gamma_1(\cN,p),\Lambda_F^1)^{\leq\underline{h}}:=\varprojlim_{r,n} M^{r}_{{\bf k}_{n}}(\Gamma_1(\cN,p),\Lambda_{n}^1)^{\leq\underline{h}},
\] 
are the spaces of \emph{overconvergent families of automorphic forms}.
By \cite[Proposition 7.8]{BM} the specialization of such families at sufficiently big classical points provide classical automorphic forms.

Let $\mu_3\in M^{\dagger}_{{\bf k}}(\Gamma_1(\cN,p),\Lambda_F^1)^{\leq\underline{h}}$ be a eigenvector for the Hecke operators not dividing $\cN$ and such that $U_\dP\mu=\alpha_3^\dP\mu$. 
Let us also consider $\mu_1,\mu_2\in M^{\dagger}_{{\bf k}}(\Gamma_1(\cN,p),\Lambda_F^1)$.

Let $\cR'=\Lambda_F^1\hat\otimes\Lambda_F^1\hat\otimes\Lambda_F'$, where $\Lambda_F'$ is the fraction field of $\Lambda_F$, thus $\cR'$ can be viewed as rational functions on $\dW^1\times\dW^1\times\dW$ with poles at finitely many weights in $\dW$. Write ${\bf k}_1$, ${\bf k}_2$ and ${\bf k}_3$ for the universal characters corresponding each copy of $\dW^1$ and $\dW$.

In the rest we use the following notation: If  $(x, y, z) \in\dW^1\times\dW^1\times\dW$ we denote by $\mu_x$, $\mu_y$ and $\mu_z$ the specializations of the families $\mu_1$ at $x$, $\mu_2$ at $y$, and $\mu_3$ at $z$ respectively.  If $z\in \dW$ is a classical weight  then the specialization $\mu_{z}$ is a classical automorphic form of weight $k_z$, where $k_z$ is the specialization of ${\bf k}_3$ at $z$. 
Recall the triple product of \eqref{gentripprod} defines a triple product
\[
t:M^{\dagger}_{{\bf k}_1}(\Gamma_1(\cN,p),\Lambda_F^1)\times M^{\dagger}_{{\bf k}_2}(\Gamma_1(\cN,p),\Lambda_F^1)\longrightarrow M^{\dagger}_{{\bf k}_3}(\Gamma_1(\cN,p),\cR')^{\leq\underline{h}}.
\]
The following result is analogous to \cite[Lemma 2.19]{DR14} and \cite[Lemma 10.3]{BM}.
\begin{lemma}\label{l:construction}
There exists $\mathcal{L}_p(\mu_1,\mu_2,\mu_3)\in\cR'$ such that for each classical point $(x, y, z) \in\dW^1\times\dW^1\times\dW$,  we have:
\[
\mathcal{L}_p(\mu_1,\mu_2,\mu_3)(x, y, z)=\frac{\left\langle\mu_z^*,t(\mu_{1},\mu_{2})_{(x,y,z)}\right\rangle}{\left\langle\mu_z^*,\mu_z\right\rangle}
\] 
where $\langle\cdot,\cdot\rangle$ is the Petersson inner product and $\mu_z^*$ defines the dual basis of $\mu_z$.
\end{lemma}

\begin{defi}
Let $\mu_1,\mu_2,\mu_3$ as above and assume that they are test vectors for three families of eigenvectors. The functions $\mathcal{L}_p(\mu_1,\mu_2,\mu_3)\in\cR'$ introduced in \ref{l:construction} is called the \emph{triple product $p$-adic L-function} of $\mu_1,\mu_2,\mu_3$. 
\end{defi}

\subsection{Interpolation formula}

Let $(\underline{r}_1,\nu_1)\in \dW^1(\C_p)$, $(\underline{r}_2,\nu_2)\in \dW^1(\C_p)$ and $\underline{r}\in \dW(\C_p)$ be classical weights and put $\underline{k}_1=k(\underline{r}_1, \nu_1), \underline{k}_2=k(\underline{r}_2, \nu_2)$ and $\underline{k}_3=k(\underline{r}_3, \nu_1+\nu_2)$. We suppose that $(\underline{k}_1,\underline{k}_2,\underline{k}_3)$ is unbalanced at $\tau_0$ with dominant weight $\underline{k}_3$ in the sense of \cite[Definition 3.4]{BM}. 
 
 We write $(x, y, z) \in \dW^1\times\dW^1\times\dW$ for the point corresponding to the triple $(\underline{k}_1,\nu_1)$, $(\underline{k}_2,\nu_2)$ and $(\underline{k}_3, \nu_1+\nu_2)$. Since $\mu_3$ has finite slope, its specialization $\mu_z$ correspond to an automorphic form of weight $(\underline{k}_3,\nu_1+\nu_2)$. If  $\underline{k}_1$ and $\underline{k}_2$ are big enough the the same is true for $\mu_x$ and $\mu_y$, obtaining automorphic forms of weights $(\underline{k}_1,\nu_1)$ and $(\underline{k}_2, \nu_2)$ respectively. We denote by $\pi_{x}$, $\pi_y$ and $\pi_z$ the automorphic representations of $(B\otimes\A_F)^{\times}$ generated  by these automorphic forms, and $\Pi_{x}$, $\Pi_y$ and $\Pi_z$ the corresponding cuspidal automorphic representations of $\mathrm{GL}_2(\A_F)$. 

Assume that $\mu_i$ are eigenvectors for all the $U_\dP$ operators, namely $U_\dP\mu_i=\alpha_{i}^{\dP}\cdot\mu_i$, and write $\alpha_{x}^{\dP}$, $\alpha_{y}^{\dP}$ and $\alpha_{z}^{\dP}$ for the corresponding specializations at $x$, $y$ and $z$. 
Moreover, we assume that $\mu_{x}$ is the test vector defined in \cite{hsieh18} of the $\mathfrak{p}$-stabilization of the newform $\mu_{x}^\circ$ for each $\dP\mid p$, and similarly for $\mu_{y}$ and $\mu_z$. Write $\beta_{i}^{\dP}$ for the other eigenvalue of $U_\dP$ as usual, and write $\beta_{x}^{\dP}$, $\beta_{y}^{\dP}$ and $\beta_{z}^{\dP}$ for the corresponding specializations. 
The following result justify the name given to $\mathcal{L}_p(\mu_1,\mu_2,\mu_3)$.
\begin{thm}\label{t:interpolation} With the notations above we have:
$$\mathcal{L}_p(\mu_1,\mu_2,\mu_3)(x, y, z)= K(\mu_{x}^\circ,\mu_{y}^\circ,\mu_{z}^\circ)\cdot\left(\prod_{\dP\mid p}\frac{\mathcal{E}_{\dP}(x,y,z)}{\mathcal{E}_{\dP,1}(z)}\right)\cdot\frac{L\left(\frac{1-\nu_1-\nu_2-\nu_3}{2},\Pi_{x}\otimes\Pi_{y}\otimes\Pi_{z}\right)^{\frac{1}{2}}}{\langle\mu_{z}^\circ,\mu_{z}^\circ\rangle}$$
here  $K(\mu_1^\circ,\mu_2^\circ,\mu_3^\circ)$ is a non-zero constant, $\mathcal{E}_\dP(x,y,z)=$
\[
\left\{\begin{array}{lc}
\mbox{\small$(1-\beta_{x}^{\dP}\beta_{y}^{\dP}\alpha_{z}^{\dP}\varpi_{\dP}^{-\underline{m}_{\dP}-\underline{2}})(1-\alpha_{x}^{\dP}\beta_{y}^{\dP}\beta_{z}^{\dP}\varpi_{\dP}^{-\underline{m}_{\dP}-\underline{2}})(1-\beta_{x}^{\dP}\alpha_{y}^{\dP}\beta_{z}^{\dP}\varpi_{\dP}^{-\underline{m}_{\dP}-\underline{2}})(1-\beta_{x}^{\dP}\beta_{y}^{\dP}\beta_{z}^{\dP}\varpi_{\dP}^{-\underline{m}_{\dP}-\underline{2}})$},&\dP\neq\dP_0\\
\mbox{\small$(1-\alpha_{x}^{\dP_0}\alpha_{y}^{\dP_0}\beta_{z}^{\dP_0}p^{1-m_{0}})(1-\alpha_{x}^{\dP_0}\beta_{y}^{\dP_0}\beta_{z}^{\dP_0}p^{1-m_{0}})(1-\beta_{x}^{\dP_0}\alpha_{y}^{\dP_0}\beta_{z}^{\dP_0}p^{1-m_{0}})(1-\beta_{x}^{\dP_0}\beta_{y}^{\dP_0}\beta_{z}^{\dP_0}p^{1-m_{0}})$},&\dP=\dP_0
\end{array}\right., 
\]
\[
\mathcal{E}_{\dP,1}(z):=\left\{\begin{array}{lc} 
(1- (\beta_z^{\dP})^2\varpi_{\dP}^{-\underline{k}_{3,\dP}-\underline{2}})\cdot (1- (\beta_z^{\dP})^2\varpi_{\dP}^{-\underline{k}_{3,\dP}-\underline{1}}),&\dP\neq \dP_0,\\
(1- (\beta_z^{\dP_0})^{2}p^{-k_{3,\tau_0}})\cdot (1- (\beta_z^{\dP_0})^{2}p^{1-k_{3,\tau_0}}),&\dP= \dP_0,
\end{array}\right. 
\]
$m_{0}=\frac{k_{1,\tau_0}+k_{2,\tau_0}+k_{3,\tau_0}}{2}\geq 0$, and $\underline{m}_{\dP}=\frac{\underline{k}_{1,\dP}+\underline{k}_{2,\dP}+\underline{k}_{3,\dP}}{2}\in\Z[\Sigma_{\mathfrak{p}}]$.
\end{thm}
\begin{proof} By Equation \eqref{espwtGM} and Lemma \ref{LemmaTripleClass} we have:
\begin{eqnarray*}
\mathcal{L}_p(\mu_1,\mu_2,\mu_3)(x, y, z)&=&\frac{\left\langle\mu_z^*,t(\mu_{1},\mu_{2})_{(x,y,z)}\right\rangle}{\left\langle\mu_z^*,\mu_z\right\rangle}=\frac{\left\langle\mu_z^*,e_{\leq h}\Theta_{m_{3,\tau_0}}(\mu_x^{[p]},\mu_y)(\underline{\Delta}_{(x,y,z)}^{\tau_0})\right\rangle}{\left\langle\mu_z^*,\mu_z\right\rangle}\\
&=&(-1)^{m_{3,\tau_0}}\binom{k_{3,\tau_0}-2}{m_{3,\tau_0}+k_{2,\tau_0}-1}^{-1}\frac{\left\langle\mu_z^*,t_{k_{1,\tau_0},k_{2,\tau_0},k_{3,\tau_0}}(\mu_x^{[p]},\mu_y)(\underline{\Delta}_{(x,y,z)}^{\tau_0})\right\rangle}{\left\langle\mu_z^*,\mu_z\right\rangle}.
\end{eqnarray*}
At this point the proof of this result is identical to that of \cite[Theorem 10.5]{BM}.
\end{proof}

	
\Addresses	
	
\bibliographystyle{plain}
\bibliography{biblio}	
	
\end{document}